%% file: main.tex
\documentclass[11pt,reqno]{amsart}
\usepackage{amssymb}
\usepackage{color,graphicx}
\usepackage{tikz}
\usepackage{times,mathptmx}
\usepackage{xspace}
\usepackage{longtable}
\usepackage[colorlinks=true,linkcolor=black,citecolor=black]{hyperref}
\usepackage[utf8]{inputenc}
\usepackage[centering,includehead,width=140mm,height=21cm]{geometry}

\input{definitions}

\begin{document}
  \input{title}

  \input{introduction} 
  \input{tropgeom} 
  \input{projections} 
  \input{realize} 
  \input{criteria} 

  \bibliographystyle{amsalpha}
  \bibliography{references}
\end{document}

%% file: definitions.tex
\theoremstyle{plain}
\newtheorem{theorem}{Theorem}[section]

\newtheorem{proposition}[theorem]{Proposition}
\newtheorem{Proposition}[theorem]{Proposition}
\newtheorem{lemma}[theorem]{Lemma}
\newtheorem{Lemma}[theorem]{Lemma}
\newtheorem{corollary}[theorem]{Corollary}
\theoremstyle{remark}
\newtheorem{construction}[theorem]{Construction}
\newtheorem{notation}[theorem]{Notation}
\newtheorem{Notations}[theorem]{Notation}
\newtheorem{example}[theorem]{Example}
\newtheorem{Example}[theorem]{Example}
\newtheorem{convention}[theorem]{Convention}
\newtheorem{remark}[theorem]{Remark}
\newtheorem{Remark}[theorem]{Remark}

\newtheorem{alg}[theorem]{Algorithm}
\theoremstyle{definition}
\newtheorem{definition}[theorem]{Definition}
\newtheorem{defn}[theorem]{Definition}

\newcommand{\NN}{{\mathbb N}}
\newcommand{\N}{\NN}
\newcommand{\PP}{{\mathbb P}}

\newcommand{\RR}{{\mathbb R}}
\newcommand{\R}{\RR}
\newcommand{\ZZ}{{\mathbb Z}}
\newcommand{\Z}{\ZZ}
\newcommand{\calF}{{\mathcal F}}
\newcommand{\calP}{{\mathcal P}}
\newcommand{\one}{{\mathbf 1}}
\newcommand{\gen}[1]{\langle#1\rangle}
\newcommand{\hooklongrightarrow}{\lhook\joinrel\longrightarrow}

\DeclareMathOperator{\cl}{cl}
\DeclareMathOperator{\ch}{char}
\DeclareMathOperator{\New}{Newt}
\DeclareMathOperator{\Newt}{Newt}
\DeclareMathOperator{\trop}{trop}
\DeclareMathOperator{\Trop}{Trop}
\DeclareMathOperator{\Spec}{Spec}
\DeclareMathOperator{\Proj}{Proj}

\DeclareMathOperator{\inom}{in}
\DeclareMathOperator{\conv}{conv}
\DeclareMathOperator{\cone}{cone}

\DeclareMathOperator{\Span}{span}

\DeclareMathOperator{\real}{Real}
\DeclareMathOperator{\realdim}{realdim}

\newcommand{\ie}{i.\,e.\ }
\newcommand{\eg}{e.\,g.\ }
\newcommand{\cf}{cf.\ }

\renewcommand{\P}{P}
\newcommand{\mult}{m}

\newlength {\sidetextwidth}
\newlength {\sidepicwidth}
\newsavebox {\sidepicbox}
\newenvironment {sidepic}[1]{%
  \par%
  \def\message##1{}%
  \hbadness=10000%
  \setlength{\sidetextwidth}{\textwidth}%
  \savebox{\sidepicbox}{\input{pics/#1}}%
  \settowidth{\sidepicwidth}{\usebox{\sidepicbox}}%
  \addtolength{\sidetextwidth}{-6mm}%
  \addtolength{\sidetextwidth}{-\sidepicwidth}%
  \begin {minipage}{\sidetextwidth}%
  \parskip1ex plus0.5ex minus0.1ex}%
  {\end {minipage}%
  \hfuzz=10000pt%
  \hspace {4mm}%
  \begin {minipage}{\sidepicwidth}\usebox{\sidepicbox}\end {minipage}%
  \hfuzz=0pt%
  \hbadness=1000%
  }

\renewenvironment {enumerate}%
  {\begin {oldenumerate}\parskip1ex plus0.5ex \itemsep 0mm \parindent 0mm}%
  {\end {oldenumerate}}

\renewenvironment {itemize}%
  {\begin {olditemize}\parskip1ex plus0.5ex \itemsep 0mm \parindent 0mm}%
  {\end {olditemize}}

\sloppypar

%% file: title.tex
\title{The Realizability of Curves in a Tropical Plane}

\author{Andreas Gathmann \and Kirsten Schmitz \and Anna Lena Winstel}

\address{Andreas Gathmann, Fachbereich Mathematik, Technische Universität
  Kaiserslautern, Postfach 3049, 67653 Kaiserslautern, Germany}
  \email{andreas@mathematik.uni-kl.de}
\address{Kirsten Schmitz, Fachbereich Mathematik, Technische Universität
  Kaiserslautern, Postfach 3049, 67653 Kaiserslautern, Germany}
  \email{schmitz@mathematik.uni-kl.de}
\address{Anna Lena Winstel, Fachbereich Mathematik, Technische Universität
  Kaiserslautern, Postfach 3049, 67653 Kaiserslautern, Germany}
  \email{winstel@mathematik.uni-kl.de}

\thanks{Kirsten Schmitz has been supported by the DFG grant Ga 636/3}
\keywords{Tropical geometry, tropicalization, tropical realizability}
  
\begin{abstract}
  Let $E$ be a plane in an algebraic torus over an algebraically closed field.
  Given a balanced $1$-dimensional fan $C$ in the tropicalization of $E$, \ie
  in the Bergman fan of the corresponding matroid, we give a complete
  algorithmic answer to the question whether or not $C$ can be realized as the
  tropicalization of an algebraic curve contained in $E$. Moreover, in the case
  of realizability the algorithm also determines the dimension of the moduli
  space of all algebraic curves in $E$ tropicalizing to $C$, a concrete simple
  example of such a curve, and whether $C$ can also be realized by an
  irreducible algebraic curve in $E$. In the first important case when $E$ is a
  general plane in a $3$-dimensional torus we also use our algorithm to prove
  some general criteria for $C$ that imply its realizability resp.\
  non-realizability. They include and generalize the main known obstructions by
  Brugallé-Shaw and Bogart-Katz coming from tropical intersection theory.
\end{abstract}

\maketitle

%% file: introduction.tex
\section{Introduction} \label{sec-introduction}

Tropical geometry is a branch of mathematics that establishes a deep connection
between algebraic geometry and combinatorics. For example, given a
$k$-dimensional subvariety $Y$ of an $n$-dimensional algebraic torus $X$ over
an algebraically closed field, the process of tropicalization assigns to it a
purely $k$-dimensional polyhedral fan $ \trop(Y) $ in an $n$-dimensional real
vector space, together with a positive integer multiplicity assigned to each
facet \cite{SP}. Although this fan is in a certain sense a simpler object than
the original variety $Y$, it still carries much information about $Y$. It is
therefore the idea of tropical geometry to study these fans by combinatorial
methods, and then transfer the results back to algebraic geometry.

In order for this strategy to work efficiently it is of course essential to
know which fans can actually occur as tropicalizations of algebraic varieties
--- this is usually called the \emph{realization problem} or \emph{tropical
inverse problem}. An important well-known necessary condition for a fan
together with given multiplicities on the facets to be realizable as the
tropicalization of an algebraic variety is the so-called \emph{balancing
condition}, certain linear relations among the multiplicities of the adjacent
cells of each codimension-$1$ cone \cite[Section 2.5]{SP}. In the case of fans
of dimension or codimension $1$ this condition is also sufficient for
realizability by an algebraic curve resp.\ hypersurface
\cite{NS06,SP,Spe07,MI}, but for intermediate dimensions no such general
statements are known so far.

Rather than considering varieties of intermediate dimension, we will restrict
ourselves in this paper to the case of curves and study a relative version of
the realization problem: let $E$ be a fixed plane in $X$, \ie a $2$-dimensional
subvariety of an algebraic torus defined by linear equations. Its
tropicalization $ \trop(E) $ is the $2$-dimensional Bergman fan of the
corresponding matroid \cite{Stu02,AK}. Given a balanced $1$-dimensional fan $C$
with rays in the support of $ \trop(E) $ --- in the following we will call this
a \emph{tropical curve} in $ \trop(E) $ --- the \emph{relative realization
problem} then is to decide whether there is a (maybe reducible) algebraic curve
$Y$ in $E$ that tropicalizes to $C$. Results in this direction are useful if
one wants to use tropical methods to analyze the geometry of algebraic curves
in (a toric compactification of) $E$, \eg for setting up moduli spaces of such
curves or studying the cone of effective curve classes.

\vspace{1ex}

\begin{sidepic}{intro-1}
  The first important example of this situation is that of a general plane $E$
  in a $3$-dimensional torus $X$. In this case (the support of) $ \trop(E) $
  will be denoted by $ L^3_2 $; it is the union of all cones generated by two
  of the classes $ [e_0],\dots,[e_3] $ of the unit vectors in $ \RR^4 /
  \gen\one $, where $ \one = (1,1,1,1) $. The picture on the right shows this
  space, together with an example of a tropical curve $C$ in it. Its rays all
  have multiplicity $1$ and are spanned by the vectors (in homogeneous
  coordinates)
    \[ [0,3,1,0],\, [0,0,1,3],\, [2,0,1,0],\, \text{and } [1,0,0,0]. \]
\end{sidepic}

Note that the balancing condition in this case just means that these four
vectors add up to $0$ in $ \RR^4 /\gen\one $. As the above representatives
(normalized so that their minimal coordinate is $0$) sum up to $ (3,3,3,3) $
any algebraic curve realizing $C$ must have degree $3$ (see Example
\ref{ex-degc} and Lemma \ref{lem-tropdeg}). We are thus asking if there is a
cubic curve in $E$ tropicalizing to $C$.

Several necessary conditions for this relative realizability have been known so
far, all of them coming from the comparison of tropical and classical
intersection theory. The strongest obstruction seems to be that of Brugallé and
Shaw, stating that an irreducible tropical curve $C$ in $ \trop(E) $ cannot be
realizable if it has a negative intersection product with another realizable
irreducible tropical curve $ D \neq C $ \cite[Corollary 3.10]{brugalleshaw},
\eg if $D$ is one of the three straight lines contained in $ L^3_2 $. They
also prove obstructions coming from the adjunction formula and intersection
with the Hessian \cite[Sections 4 and 5]{brugalleshaw}. In addition, Bogart and
Katz have shown that a tropical curve in $E$ contained in a classical
hyperplane can only be realizable if it contains a classical line or is a
multiple of the tropical intersection product of $E$ with this hyperplane
\cite[Proposition 1.3]{bogartkatz}. However, none of these criteria are also
sufficient for realizability. They all fail to detect some of the
non-realizable curves --- \eg the curve $C$ in $ L^3_2 $ in the picture above,
which actually turns out to be non-realizable by an algebraic curve in $E$ (see
Proposition \ref{prop:commonray} and Example \ref{ex-table}).

In this paper we will take a different approach to the relative realization
problem. It is of an algorithmic nature, and thus first of all leads to an
efficient way to decide for any given tropical curve $C$ in $ \trop(E) $
whether or not it is realizable by an algebraic curve in $E$. After recalling
the basic tropical background in section \ref{sec-tropgeom}, we then show in
sections \ref{sec-projections} and \ref{sec-realize} that checking whether the
tropicalization of an algebraic curve is equal to $C$ is equivalent to checking
that the projections of the curve to the various coordinate planes tropicalize
to the corresponding projections of $C$ to $ \RR^3 / \gen\one $. As these
checks are now in the plane, they can easily be performed explicitly by
comparing Newton polytopes. The resulting Algorithm \ref{algorithm} to decide
for relative realizability is also available for download as a Singular library
\cite{Sing,Win12}. It can distinguish between realizability by a reducible and
by an irreducible curve, compute the dimension of the space of algebraic curves
tropicalizing to $C$ (which in fact is an open subset of a linear space), and
provide an explicit easy example of such an algebraic curve in case of
realizability. The computations can be performed for ground fields of any
characteristic, and in fact the results will in general depend on this choice
(see Example \ref{ex-char}).

From the numeric results of these computations it seems unlikely that there is
a general easy rule to decide for realizability in any given case. However, by
a systematic study of the algorithm we prove some criteria in section
\ref{sec-criteria} that imply realizability resp.\ non-realizability in many
cases of interest. In the case of $ L^3_2 $ they include and generalize the
main previously known obstructions by Brugallé-Shaw and Bogart-Katz mentioned
above, thus putting them into a common framework with a unified idea of proof
(see Propositions \ref{prop:intprod} and \ref{prop:bogartkatz}). In addition,
our criteria show that every tropical curve in $ \trop(E) $ can be realized by
an algebraic cycle in $E$, \ie by a formal $ \ZZ $-linear combination of
algebraic curves in $E$ (see Proposition \ref{prop-realize-cycle}).

\begin{sidepic}{intro-2}
  One example of a new obstruction to realizability in the case of tropical
  curves in $ L^3_2 $ is shown in the picture on the right: a tropical curve
  that is completely contained in the shaded area cannot be realizable by an
  algebraic curve in $E$ if its multiplicity on the ray $ [e_0] $ is $1$ ---
  regardless of the characteristic of the ground field (see Proposition
  \ref{prop:commonray}). This shows \eg the non-realizability of the example
  curve that we had considered in the picture above.
\end{sidepic}

The following numbers may be useful to get a feeling for the numerical
complexity of the problem: there are (up to coordinate permutations) $182$
tropical curves of degree $3$ and $2122$ curves of degree $4$ in $ L^3_2 $. In
characteristic zero, $17$ of the degree $3$-curves and $138$ of the degree-$4$
curves are not realizable. Checking the realizability of all these curves takes
less than one minute on a standard PC. In degree $3$ our general criteria
suffice to find all non-realizable curves, whereas $21$ of the $138$
non-realizable curves remain undetected by these obstructions in degree $4$
(see Example \ref{ex-table}).

It should be noted that the methods of this paper are quite general and can
also be applied \eg to the ``non-constant coefficient case'', \ie to the
question which $1$-dimensional balanced polyhedral complexes in $ \trop(E) $
can be realized as the tropicalization of an algebraic curve in $E$ over a
non-Archimedean valued field. Work in this direction is in progress.

%% file: tropgeom.tex
\section{Tropical Geometry} \label{sec-tropgeom}

We will start by recalling the basic combinatorial concepts from tropical
geometry used in this paper. More details can be found \eg in \cite{AR}.

\begin{notation}[Tropical cycles] \label{not-cycles}
  Let $n\in \N$, let $ \Lambda $ be a lattice of rank $n$, and let $ V =
  \Lambda \otimes_\ZZ \RR$ be the corresponding real vector space. By a
  \emph{cone} $ \sigma $ in $V$ we will always mean a rational polyhedral cone.
  Let $ V_\sigma \subset V $ be the vector space spanned by $ \sigma $, and $
  \Lambda_\sigma := V_\sigma \cap \Lambda $. If the cone $ \tau $ is a face of
  $ \sigma $ of codimension $1$, we denote by $ u_{\sigma/\tau} \in
  \Lambda_\sigma / \Lambda_\tau $ the primitive normal vector of $ \sigma $
  modulo $ \tau $. In the case $ \dim \sigma = 1 $ we write $ u_{\sigma/\{0\}}
  \in \Lambda_\sigma \subset \Lambda $ also as $ u_\sigma $.

  For $ r \in \NN $ an \emph{$r$-dimensional (tropical) cycle} or
  \emph{$r$-cycle} in $V$ is a pure $r$-dimensional fan $C$ of cones in $V$ as
  above, together with a \emph{multiplicity} $ m_C(\sigma) \in \ZZ $ for each
  maximal cone $ \sigma \in C $, and such that the \emph{balancing condition}
    \[ \sum_{\sigma > \tau} m_C(\sigma) \, u_{\sigma/\tau} = 0
       \quad \in V/V_\tau \]
  holds for each $ (r-1) $-dimensional cone $ \tau \in C $ (where the sum is
  taken over all maximal cones $ \sigma $ containing $ \tau $ as a face). If
  there is no risk of confusion, we will also write $ m(\sigma) $ instead of $
  m_C(\sigma) $. A tropical cycle with only non-negative multiplicities will be
  called a \emph{tropical variety}, resp.\ a \emph{tropical curve} if $ r=1 $.

  The \emph{support} $ |C| \subset V $ of a tropical cycle $C$ is the union of
  its maximal cones that have non-zero multiplicity. If $D$ is another tropical
  cycle in $V$ with $ |D| \subset |C| $ we say that $D$ is contained in $C$,
  and also write this as $ D \subset C $ by abuse of notation. The abelian
  group of all $k$-dimensional cycles contained in $C$, modulo refinements as
  in \cite[Definition 2.12]{AR}, will be denoted by $ Z_k^{\trop}(C) $.
\end{notation}

\begin{construction}[Intersection products]
    \label{constr-intersection}
  A \emph{rational function} on a $k$-dimensional cycle $C$ is a continuous
  piecewise integer linear function $ \varphi: |C|\rightarrow \mathbb{R} $,
  where we will assume the fan structure of $C$ to be fine enough so that $
  \varphi $ is linear on each cone $ \sigma $, see \cite[Definition 3.1]{AR}.
  This linear function, extended uniquely to $ V_\sigma $, will be denoted $
  \varphi_\sigma $. We then define the \emph{intersection product} $ \varphi
  \cdot C \in Z_{k-1}^{\trop}(C) $ to be the cycle whose maximal cones are the
  $ (k-1) $-dimensional cones $ \tau $ of $C$ with multiplicities
    \[ m_{\varphi \cdot C}(\tau) = \varphi_\tau
         \left( \sum_{\sigma>\tau} m_C(\sigma) \, v_{\sigma/\tau} \right)
         - \sum_{\sigma>\tau} m_C(\sigma)\,\varphi_\sigma(v_{\sigma/\tau}), \]
  where the sum is taken over all $k$-dimensional cones $ \sigma $ in $C$
  containing $ \tau $ as a face, and the vectors $ v_{\sigma/\tau} $ are
  arbitrary representatives of $ u_{\sigma/\tau} $ \cite[Definition 3.4]{AR}.
  Its support is contained in the locus of points at which $ \varphi $ is not
  locally linear.

  If $ \varphi \cdot V = D $ we say that the rational function $ \varphi $ cuts
  out $D$, and write the intersection product $ \varphi \cdot C $ also as $ D
  \cdot C $. This intersection product of a codimension-$1$ cycle $D$ with $C$
  is well-defined (\ie independent of the rational function cutting out $D$),
  and satisfies the expected properties as \eg commutativity if $C$ can also be
  cut out by a rational function \cite[Section 9]{AR}.
\end{construction}

\begin{construction}[Push-forward of cycles] \label{constr-push}
  Let $ f: \Lambda \to \Lambda' $ be a linear map of lattices. By abuse of
  notation, the corresponding linear map of vector spaces $ V = \Lambda
  \otimes_\ZZ \RR \to V' = \Lambda' \otimes_\ZZ \RR $ will also be denoted by
  $f$. For $ C \in Z_k^{\trop}(V) $ there is then an associated
  \emph{push-forward} cycle $ f_*(C) \in Z_k^{\trop}(V') $ obtained as follows:
  subdivide $C$ so that the collection of cones $ \{ f(\sigma): \sigma \in C \}
  $ is a fan in $V'$, and associate to each such image cone $ \tau $ of
  dimension $k$ the multiplicity
    \[ m_{f_* C}(\tau) = \sum_{\sigma: f(\sigma) = \tau} m_C(\sigma) \cdot
       [\Lambda'_\tau : f(\Lambda_\sigma)]. \]
  This way one indeed obtains a balanced cycle, and the corresponding
  push-forward map $ f_*: Z_k^{\trop}(V) \to Z_k^{\trop}(V') $ is a
  homomorphism that satisfies all expected properties as \eg the projection
  formula \cite[Section 4]{AR}.
\end{construction}

\begin{convention}[Homogeneous coordinates] \label{conv-homo}
  In the following, we will always work with real vector spaces that have fixed
  homogeneous coordinates, \ie we have $ V = \RR^N/\gen\one $ for a finite
  index set $N$, where $ \one $ denotes the vector all of whose coordinates are
  equal to $1$. It is then always understood that the underlying lattice is $
  \ZZ^N/\gen\one $. The class of a vector $ v \in \RR^N $ in $ \RR^N/\gen\one $
  will be written $ [v] $; for $ i \in N $ the unit vector in $ \ZZ^N $ with
  entry $1$ in the coordinate $i$ is denoted by $ e_i $. Often we will just
  have $ N=\{0,\dots,n\} $, in which case we write $V$ as $ \RR^{n+1}/\gen\one
  $ with lattice $ \ZZ^{n+1}/\gen\one $.

  The reason for this choice is that these are the natural ambient spaces for
  \emph{matroid fans} --- tropical varieties that will be central in this paper
  as they occur as tropicalizations of linear spaces \cite[Theorem 1]{AK}.
  Let us now introduce these matroid fans from a combinatorial point of view.
  Details on matroid theory can be found in \cite{Oxl}.
\end{convention}

\begin{construction}[Matroid fans] \label{constr-matroid}
  Let $M$ be a loop-free matroid on a finite ground set $N$. By a \emph{chain
  of flats} (of length $m$) in $M$ we will mean a sequence $ \calF =
  (F_1,\dots,F_m) $ of flats of $M$ with
    \[ \emptyset \subsetneq F_1 \subsetneq F_2 \subsetneq \cdots \subsetneq F_m
       \subsetneq N. \]
  For such a chain of flats let $ \sigma_\calF \subset \RR^N/\gen\one $ be the
  $m$-dimensional simplicial cone generated by the classes of the vectors $
  v_{F_1},\dots,v_{F_m} $, where $ v_F \in \RR^N $ for a flat $F$ denotes the
  vector with entries $1$ in the coordinates of $F$, and $0$ otherwise. One can
  show that the collection of all cones $ \sigma_\calF $ corresponding to
  chains of flats in $M$, with multiplicity $1$ assigned to each maximal cone,
  is a tropical variety of dimension equal to the rank of $M$ minus $1$
  \cite[Proposition 3.1.10]{Fra12}. It is called the \emph{matroid fan} or
  \emph{Bergman fan} associated to $M$ and denoted by $ B(M) $.
\end{construction}

\begin{sidepic}{ex-l32} \vspace{-6pt} \begin{example}[General linear spaces
    $ L^n_k $] \label{ex-matroid}
  Let $ n,k \in \NN $ with $ k \le n $, and let $M$ be the uniform matroid
  of rank $ k+1 $ on $ N = \{0,\dots,n\} $. Then the matroid fan $ B(M) $
  consists of the cones spanned by the vectors $ [v_{F_1}],\dots,[v_{F_m}] $
  for all sequences $ \emptyset \subsetneq F_1 \subsetneq \cdots \subsetneq F_m
  \subsetneq N $ with $ |F_m| \le k $. We denote it by $ L^n_k $; the picture
  on the right shows the case of $ L^3_2 $ in $ \RR^4/\gen\one $ (see also
  Example \ref{ex-rank}).

  In the following we will consider tropical cycles only up to refinements.
  Hence, we will often draw $ L^3_2 $ without the subdivision induced by the
  rank-$2$ flats.
\end{example} \end{sidepic}

Our main tropical objects in this paper will be tropical curves in matroid
fans. So let us now introduce some convenient notations to deal with such
curves.

\begin{notation}[Description of a curve $C$ with the set $ \P(C) $]
    \label{notation-pofc}
  For a tropical curve $C$ in $ \RR^N / \gen\one $ we will always assume that
  it is subdivided so that the origin is a cone of $C$. If $ \sigma_1,\dots,
  \sigma_k $ are the $1$-dimensional cones of $C$, we set
    \[ \P(C) := \{ m(\sigma_1) \, v_1,\dots, m(\sigma_k) \, v_k \}
       \quad \subset \ZZ^N, \]
  where $ v_i \in \ZZ^N $ for $ i=1,\dots,k $ is the unique representative of
  the primitive normal vector $ u_{\sigma_i} \in \ZZ^n / \gen\one $ such that
  the minimum over all its coordinates is $0$. For $ v \in \ZZ^N $ denote by $
  \gcd(v) $ the (non-negative) greatest common divisor of the coordinates of
  $v$. Then $ \gcd(v_i) = 1 $, and so for all $i$ we have $ \gcd(m(\sigma_i) \,
  v_i) = m(\sigma_i) $ and $ [m(\sigma_i) \, v_i] \in \sigma_i $. This means
  that the set $ \P(C) $ allows to reconstruct the curve $C$ uniquely, and thus
  is a convenient way to describe curves in $ \RR^N / \gen\one $. By abuse of
  notation, we will write the multiplicity $ m(\sigma_i) $ also as $ m(v_i) $
  or $ m([v_i]) $.
\end{notation}

We will now introduce the degree of a tropical $1$-cycle and show that the set
$ \P(C) $ gives a convenient way to compute it in the case of curves.

\begin{definition}[Degree of a tropical $1$-cycle] \label{def-degree}
  The \emph{degree} $ \deg(C) $ of a tropical $1$-cycle $C$ in $ \RR^{n+1} /
  \gen\one $ is defined to be the (multiplicity of the origin in the)
  intersection product $ L^n_{n-1} \cdot C $ of $C$ with a general tropical
  hyperplane. 
\end{definition}

\begin{lemma}[The degree in terms of $ \P(C) $] \label{lem-degree}
  Let $ C \subset \RR^{n+1} / \gen\one $ be a tropical curve. Then $ \sum_{v
  \in \P(C)} v = \deg(C) \cdot \one $.
\end{lemma}

\begin{proof}
  As $ L^n_{n-1} $ is cut out by the function $ \varphi(x) =
  \min(x_0-x_0,x_1-x_0,\dots,x_n-x_0) $, the intersection product $ L^n_{n-1}
  \cdot C $ is easily computed with the formula of Construction
  \ref{constr-intersection} for $ \tau=\{0\} $: the first term vanishes due to
  the balancing condition, and thus every $1$-dimensional cone $ \sigma $ in
  $C$ with corresponding vector $ (x_0,\dots,x_n) $ in $ P(C) $, \ie such that
  $ m(\sigma) \, u_\sigma = [x_0,\dots,x_n] $ and $ \min (x_0,\dots,x_n) = 0 $,
  gives rise to a contribution of
    \[ - m(\sigma) \, \varphi (u_\sigma)
       = - \min (x_0-x_0,x_1-x_0,\dots,x_n-x_0) = x_0 \]
  to $ L^n_{n-1} \cdot C $. In other words, the first coordinates of all
  vectors in $ P(C) $ sum up to $ \deg(C) $. Of course, by symmetry this means
  that the sum of all vectors in $ P(C) $ is $ \deg(C) \cdot \one $.
\end{proof}

\begin{example} \label{ex-degc}
  For the tropical curve $C$ in $ L^3_2 $ from the picture in the introduction
  we have
    \[ \P(C) = \{ (0,3,1,0), (0,0,1,3), (2,0,1,0), (1,0,0,0) \}
       \quad \subset \ZZ^4. \]
  As these vectors sum up to $ (3,3,3,3) $, we see by Lemma \ref{lem-degree}
  that $C$ has degree $3$.
\end{example}

For our applications we will need intersection products of $ L^3_2 $ with a
classical plane. For this, let $ a_0,a_1,a_2,a_3 \in \ZZ $ not all zero with $
a_0+\cdots+a_3=0 $. We set $ f: \RR^4 / \gen\one \to \RR, \; (x_0,\dots,x_3)
\mapsto a_0x_0 + \cdots + a_3x_3 $ and $ \varphi: \RR^4 / \gen\one \to \RR, \;
x \mapsto \min(0,f(x)) $. Then the rational function $ \varphi $ cuts out a
cycle $H$ whose support is just the classical plane given by the equation $ f=0
$. We want to compute the intersection cycle $ L^3_2 \cdot H $.

\begin{lemma}[Intersection products of $ L^3_2 $ with classical planes]
    \label{lem-plane}
  With the notations as above, let $ d = \sum_{i: a_i>0} \, a_i $. Then the set $
  P(C) $ for $ C = L^3_2 \cdot H $ consists exactly of the following vectors:
  \begin{enumerate}
  \item \label{lem-plane-a}
    $ a_i e_j - a_j e_i $ for all $ i,j \in \{0,1,2,3\} $ with $ a_i>0 $
    and $ a_j<0 $;
  \item \label{lem-plane-b}
    $ d \, e_i $ for all $ i \in \{0,1,2,3\} $ with $ a_i = 0 $.
  \end{enumerate}
  In particular, we have $ \deg(C) = d $.
\end{lemma}

\def\message#l32-plane{}%
  {\begin{center}\input{pics/l32-plane}\end{center}}

\begin{proof}
  By Construction \ref{constr-intersection}, the cones that can occur in the
  intersection product $ L^3_2 \cdot H $ are the $1$-dimensional cones of $
  L^3_2 \cap H $. As these are exactly the classes of the vectors listed in the
  lemma, it only remains to compute their multiplicities in $ L^3_2 \cdot
  H $.

  Moreover, for all possibilities of the signs of $ a_0,\dots,a_3 $, the
  vectors listed in the lemma sum up to $ (d,d,d,d) $. Hence, they form a
  balanced cycle of degree $d$ by Lemma \ref{lem-degree}. As shown in the
  picture above, the number of these vectors can vary (3 or 4), but in any case
  at most two of them are of type \ref{lem-plane-b}, \ie point along a ray of $
  L^3_2 $. Since $ L^3_2 \cdot H $ is a balanced cycle too and the balancing
  condition in the plane $H$ allows to reconstruct the multiplicities of up to
  two linearly independent cones, it thus suffices to check the multiplicities
  in the case \ref{lem-plane-a}.

  In this case we can assume by symmetry that $ i=0 $ and $ j=1 $. Locally
  around the $1$-dimensional cone $ \sigma $ spanned by $ (-a_1,a_0,0,0) $,
  the cycles $ L^3_2 $ and $H$ are then cut out by the rational functions $
  \min(0,x_2-x_3) $ and $ \min(0,a_0x_0 + \cdots + a_3x_3) $, respectively. In
  non-homogeneous coordinates with $ x_3 = 0 $ the corresponding functions are
  $ \min(0,x_2) $ and $ \min(0,a_0x_0+a_1x_1 + a_2 x_2) $. By
  \cite[Lemma 1.4]{Rau08} the multiplicity of $ \sigma $ in the intersection
  product is therefore the index of the lattice $ \{ (x_2, a_0x_0 + a_1x_1 +
  a_2 x_2): x_0,x_1,x_2 \in \ZZ \} $ in $ \ZZ^2 $, \ie the (positive) greatest
  common divisor of the $ 2 \times 2 $ minors of the matrix
    \[ \begin{pmatrix} 0 & 0 & 1 \\ a_0 & a_1 & a_2 \end{pmatrix}, \]
  which is just $ \gcd(a_0,a_1) $. As desired, the vector in $ P(C) $
  corresponding to the cone $ \sigma $ is thus $ \gcd(a_0,a_1) \cdot \frac
  1{\gcd(a_0,a_1)} (-a_1,a_0,0,0) = (-a_1,a_0,0,0) $.
\end{proof}

\begin{construction}[Intersection products in $ L^3_2 $] \label{constr-l32}
  Intersection products of cycles can not only be constructed in vector spaces,
  but also in matroid fans \cite{FJ10,Sha10}. In this paper we will only need
  the (degree of the) intersection product of two curves $ C_1 $ and $ C_2 $ in
  $ L^3_2 $; by \cite[Proposition 4.1]{Sha10} it is given by the explicit
  formula
    \[ C_1 \cdot C_2 = \deg(C_1) \cdot \deg(C_2)
       - \sum_{0 \le i<j \le 3} \;\;
         \sum_{\substack{a\,e_i+b\,e_j \in P(C_1) \\ a,b>0}} \;\;
         \sum_{\substack{c\,e_i+d\,e_j \in P(C_2) \\ c,d>0}} \;\;
         \min (ad,bc). \]
\end{construction}

%% file: pics/l32-plane.tex
\begin{picture}(0,0)%
\includegraphics{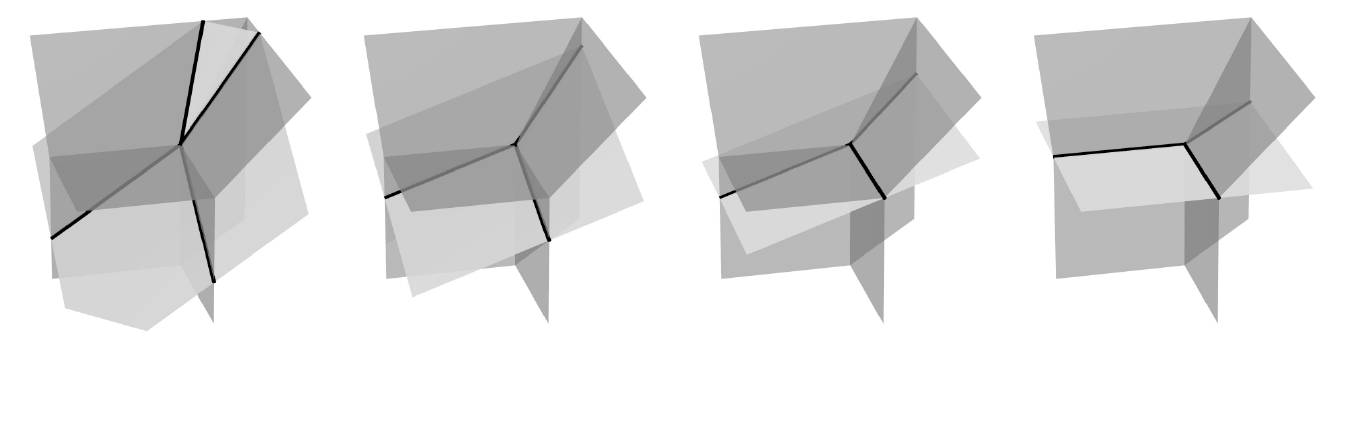}%
\end{picture}%
\setlength{\unitlength}{4144sp}%
\begingroup\makeatletter\ifx\SetFigFont\undefined%
\gdef\SetFigFont#1#2#3#4#5{%
  \reset@font\fontsize{#1}{#2pt}%
  \fontfamily{#3}\fontseries{#4}\fontshape{#5}%
  \selectfont}%
\fi\endgroup%
\begin{picture}(6211,1909)(451,-1970)
\put(1261,-1726){\makebox(0,0)[b]{\smash{{\SetFigFont{10}{12.0}{\familydefault}{\mddefault}{\updefault}{\color[rgb]{0,0,0}all $a_i\neq 0$,}%
}}}}
\put(1261,-1906){\makebox(0,0)[b]{\smash{{\SetFigFont{10}{12.0}{\familydefault}{\mddefault}{\updefault}{\color[rgb]{0,0,0}exactly two $ a_i>0 $}%
}}}}
\put(2791,-1726){\makebox(0,0)[b]{\smash{{\SetFigFont{10}{12.0}{\familydefault}{\mddefault}{\updefault}{\color[rgb]{0,0,0}all $a_i\neq 0$,}%
}}}}
\put(2791,-1906){\makebox(0,0)[b]{\smash{{\SetFigFont{10}{12.0}{\familydefault}{\mddefault}{\updefault}{\color[rgb]{0,0,0}one or three $ a_i>0 $}%
}}}}
\put(4321,-1726){\makebox(0,0)[b]{\smash{{\SetFigFont{10}{12.0}{\familydefault}{\mddefault}{\updefault}{\color[rgb]{0,0,0}exactly one $a_i = 0$}%
}}}}
\put(5851,-1726){\makebox(0,0)[b]{\smash{{\SetFigFont{10}{12.0}{\familydefault}{\mddefault}{\updefault}{\color[rgb]{0,0,0}exactly two $a_i = 0$}%
}}}}
\put(1891,-556){\makebox(0,0)[lb]{\smash{{\SetFigFont{10}{12.0}{\familydefault}{\mddefault}{\updefault}{\color[rgb]{0,0,0}$L^3_2$}%
}}}}
\put(3421,-556){\makebox(0,0)[lb]{\smash{{\SetFigFont{10}{12.0}{\familydefault}{\mddefault}{\updefault}{\color[rgb]{0,0,0}$L^3_2$}%
}}}}
\put(4951,-556){\makebox(0,0)[lb]{\smash{{\SetFigFont{10}{12.0}{\familydefault}{\mddefault}{\updefault}{\color[rgb]{0,0,0}$L^3_2$}%
}}}}
\put(6481,-556){\makebox(0,0)[lb]{\smash{{\SetFigFont{10}{12.0}{\familydefault}{\mddefault}{\updefault}{\color[rgb]{0,0,0}$L^3_2$}%
}}}}
\put(1801,-1231){\makebox(0,0)[lb]{\smash{{\SetFigFont{10}{12.0}{\familydefault}{\mddefault}{\updefault}{\color[rgb]{0,0,0}$H$}%
}}}}
\put(3331,-1141){\makebox(0,0)[lb]{\smash{{\SetFigFont{10}{12.0}{\familydefault}{\mddefault}{\updefault}{\color[rgb]{0,0,0}$H$}%
}}}}
\put(4861,-961){\makebox(0,0)[lb]{\smash{{\SetFigFont{10}{12.0}{\familydefault}{\mddefault}{\updefault}{\color[rgb]{0,0,0}$H$}%
}}}}
\put(6391,-1074){\makebox(0,0)[lb]{\smash{{\SetFigFont{10}{12.0}{\familydefault}{\mddefault}{\updefault}{\color[rgb]{0,0,0}$H$}%
}}}}
\end{picture}%

%% file: projections.tex
\section{Projections of Matroid Fans} \label{sec-projections}

In order to study the (relative) realizability of tropical curves in
$2$-dimensional matroid fans, our strategy is to use coordinate projections to
map the situation to the plane, where we can then apply Newton polytope
techniques. For example, there are four projections of the space $ L^3_2
\subset \RR^4/\gen\one $ of Example \ref{ex-matroid} to the plane $
\RR^3/\gen\one $ which are described by forgetting one of the coordinates. But
of course none of these projections is injective, and thus all of them lose
some information on the curves in $ L^3_2 $. It is the main goal of this
section to prove that all coordinate projections together suffice to
reconstruct arbitrary tropical curves in the matroid fan (see Corollary
\ref{cor-reconstruct}).

Throughout this section, let $M$ be a loop-free matroid on a finite ground set
$N$, and let $ B(M) \subset \RR^N / \gen\one $ be the corresponding matroid fan
as in Construction \ref{constr-matroid}, consisting of all cones $ \sigma_\calF
$ for chains of flats $ \calF = (F_1,\dots,F_m) $ in $M$. Recall that $
\sigma_\calF $ is generated by the vectors $ [v_{F_i}] $, where $ v_F \in \RR^N
$ for a flat $F$ has $i$-th coordinate $1$ for $ i \in F $, and $0$ for $ i
\notin F $. For details on matroid theory we refer to \cite{Oxl}.

\begin{construction}[Projections of matroid fans] \label{constr-project}
  For a non-empty subset $ A \subset N $ we denote by $ p^A: \RR^N \to \RR^A
  $ (and by abuse of notation also $ p^A: \RR^N / \gen\one \to \RR^A /
  \gen\one $) the projection onto the coordinates of $A$. Our goal is to
  describe the projection $ p^A (B(M)) $.

  For this we consider the so-called \emph{restricted matroid} $ M|_A $ on $A$
  whose independent sets are exactly those subsets of $A$ that are independent
  subsets of $N$ in $M$. It gives rise to a matroid fan $ B(M|_A) \subset \RR^A
  / \gen\one $. We will now show that $ p^A $ maps $ B(M) $ to $ B(M|_A) $, and
  describe this map more precisely. An example of this is shown in the picture
  below, where $M$ is the uniform rank-$3$ matroid on $ N = \{0,1,2,3\} $, so
  that $ B(M) = L^3_2 $, and $ A = \{0,1,2\} $. Hence, $ M|_A $ is the uniform
  rank-$3$ matroid on $ \{0,1,2\} $, the map $ p^A $ just forgets the last
  coordinate, and can be viewed in the picture as the vertical projection onto
  $ \RR^A / \gen\one \cong \RR^2 $.

  \def\message#ex-matroid{}%
  {\begin{center}\input{pics/ex-matroid}\end{center}}

  In order to describe $ p^A $ we first note that, if $ F \subset N $ is a
  flat of $M$, then $ F \cap A \subset A $ is a flat of $ M|_A $
  \cite[Proposition 3.3.1]{Oxl}. So if $ \calF = (F_1,\dots,F_m) $ is a chain
  of flats in $M$ then $ (F_1 \cap A,\dots,F_m \cap A) $ is a collection of
  ascending flats in $A$ --- it might be however that some of these flats
  coincide or are equal to $ \emptyset $ or $A$. We denote by $ \calF \cap A $
  the chain of flats in $ M|_A $ obtained from the sequence $ (F_1 \cap A,
  \dots,F_m \cap A) $ by deleting repeated entries and those that are equal to
  $ \emptyset $ or $A$. In our example in the picture above, the chain of flats
  $ \calF = (\{0\},\{0,3\}) $ in $M$ would \eg give rise to the chain of
  flats $ \calF \cap A = (\{0\}) $ in $ M|_A $.

  With these notations we can now describe the projection $ p^A $ as follows.
\end{construction}

\begin{lemma}[Properties of projections of matroid fans] \label{lem-project}
  Let $ A \subset N $ be a non-empty subset. With notations as in Constructions
  \ref{constr-matroid} and \ref{constr-project}, we have for the corresponding
  projection $ p^A: \RR^N / \gen\one \to \RR^A / \gen\one $:
  \begin{enumerate}
  \item \label{lem-project-a}
    $ p^A ([v_F]) = [v_{F \cap A}] $ for every flat $F$ of $M$.
  \item \label{lem-project-b}
    Let $ \calF $ be a chain of flats in $M$. Then $ p^A $ maps the
    corresponding cone $ \sigma_\calF $ of $ B(M) $ surjectively to the cone
    $ \sigma_{\calF \cap A} $ of $ B(M|_A) $. The map $p^A|_{\sigma_\calF}$ is
    bijective if and only if the chains $ \calF $ and $ \calF \cap A $ have
    the same length.
  \item \label{lem-project-c}
    The maps $p^A$ and $p^A|_{\sigma_\calF}$ of \ref{lem-project-b} are also
    surjective and bijective, respectively, over $ \ZZ $, \ie they map $
    V_{\sigma_\calF} \cap (\ZZ^N / \gen\one) $ surjectively and bijectively,
    respecively, to $ V_{\sigma_{\calF \cap A}} \cap (\ZZ^A / \gen\one) $.
  \item \label{lem-project-d}
    $ p^A $ maps $ B(M) $ surjectively to $ B(M|_A) $.
  \item \label{lem-project-e}
    $ p^A $ maps $ B(M) $ surjectively to $ \RR^A / \gen\one $ if and only
    if $A$ is an independent set in $M$.
  \end{enumerate}
\end{lemma}

\begin{proof}
  Statement \ref{lem-project-a} follows immediately from the definition of $
  v_F $, since $ p^A $ just forgets the coordinates of $A$. So if $ \calF =
  (F_1,\dots,F_m) $ is a chain of flats in $M$, the cone $ \sigma_\calF $
  spanned by $ [v_{F_1}],\dots,[v_{F_m}] $ is mapped by $ p^A $ surjectively
  to the cone spanned by $ [v_{F_1 \cap A}],\dots,[v_{F_m \cap A}] $, which by
  definition is equal to $ \sigma_{\calF \cap A} $. As a linear map of cones
  this map is bijective if and only if $ \sigma_\calF $ and $ \sigma_{\calF
  \cap A} $ have the same dimension, \ie if $ \calF $ and $ \calF \cap A $
  have the same length. This shows \ref{lem-project-b}. Statement
  \ref{lem-project-c} follows in the same way, noting that the lattices
  $ V_{\sigma_\calF} \cap (\ZZ^N / \gen\one) $ and $ V_{\sigma_{\calF \cap A}}
  \cap (\ZZ^A / \gen\one) $ are spanned by the classes of $ v_{F_1},\dots,
  v_{F_m} $ and $ v_{F_1 \cap A},\dots,v_{F_m \cap A} $, respectively.

  To show the last two statements, note that for every chain of flats $
  \calF' = (F'_1,\dots,F'_m) $ in $ M|_A $ we get a chain of flats $ \calF
  = (\cl(F'_1),\dots,\cl(F'_m)) $ in $M$ with $ \calF \cap A = \calF' $ by
  applying the closure operator $ \cl $ of $M$ \cite[3.1.16]{Oxl}. Thus $
  p^A(\sigma_\calF) = \sigma_{\calF'} $, and hence, the image of $ p^A $ is
  all of $ B(M|_A) $, as claimed in \ref{lem-project-d}. As $ \RR^A /
  \gen\one $ is irreducible \cite[chapter 2]{GKM}, this image $ B(M|_A) $ is
  equal to $ \RR^A / \gen\one $ if and only if its dimension is equal to $
  |A|-1 $. This is the case if and only if the matroid $ M|_A $ has rank $ |A|
  $, which in turn is equivalent to saying that $ M|_A $ is the uniform matroid
  on $A$, \ie that $A$ is an independent set in $M$. This proves
  \ref{lem-project-e}.
\end{proof}

\begin{definition}[Rank of a chain of flats] \label{def-rank}
  Let $ r: \calP(N) \to \NN $ be the rank function of the matroid $M$, \cf 
  \cite[Section 1.3]{Oxl}.
  For a chain of flats $ \calF = (F_1,\dots,F_m) $ of $M$, with $ \emptyset
  \subsetneq F_1 \subsetneq \cdots \subsetneq F_m \subsetneq N $ as above, we
  define the \emph{rank} of $ \calF $ to be
    \[ r(\calF) := r(F_1) + \cdots + r(F_m). \]
  We will also call this the rank of the corresponding cone $ \sigma_\calF $ of
  $ B(M) $.
\end{definition}

\begin{lemma} \label{lem-rank}
  For each chain of flats $ \calF $ of $M$ there is a basis $ A \subset N $ of
  $M$ such that:
  \begin{enumerate}
  \item \label{lem-rank-a}
    the projection $ p^A $ is injective on the cone $ \sigma_\calF $ of $
    B(M) $;
  \item \label{lem-rank-b}
    for every other chain of flats $ \calF' \neq \calF $ of $M$ with $
    p^A(\sigma_{\calF'}) = p^A(\sigma_\calF) $ we have $ r(\calF') >
    r(\calF) $.
  \end{enumerate}
\end{lemma}

\begin{proof}
  Extend $ \calF=(F_1,\dots,F_m) $ to a maximal chain of flats $
  (G_1,\dots,G_k) $ of length $ k := r(N)-1 $ (with $ G_0 := \emptyset
  \subsetneq G_1 \subsetneq \cdots \subsetneq G_k \subsetneq N =: G_{k+1} $).
  Choosing an element of each $ G_i \backslash G_{i-1} $ for $ i=1,\dots,k+1 $,
  we then obtain a basis $A$ of $M$ with $ r(G_i \cap A) = r(G_i) $ for all $
  i=1,\dots,k $, and thus $ r(F_i) = r(F_i \cap A) $ for all $ i=1,\dots,m $.

  Hence, from $ 1 \le r(F_1) < \cdots < r(F_m) \le k $ it follows that
  $ 1 \le r(F_1 \cap A) < \cdots < r(F_m \cap A) \le k $. So the sets $ F_1
  \cap A,\dots,F_m \cap A $ are all distinct and not equal to $ \emptyset $ or
  $A$, and thus $ \calF \cap A = (F_1 \cap A,\dots,F_m \cap A) $ and $
  r(\calF \cap A) = r(\calF) $. In particular, since $ \calF $ and $ \calF \cap
  A $ have the same length it follows from Lemma \ref{lem-project}
  \ref{lem-project-b} that $ p^A $ is injective on $ \sigma_\calF $. This
  shows part \ref{lem-rank-a} of the lemma.

  Now let $ \calF' = (F'_1,\dots,F'_q) $ be another chain of flats with the
  same image under $ p^A $, \ie by Lemma \ref{lem-project}
  \ref{lem-project-b} such that $ \calF' \cap A = \calF \cap A = (F_1 \cap A,
  \dots,F_m \cap A) $. Then for each $ i=1,\dots,m $ there must be an index $
  j_i \in \{1,\dots,q\} $ with $ F'_{j_i} \cap A = F_i \cap A $, and thus
    \[ r(\calF') \ge \sum_{i=1}^m r(F'_{j_i})
         \ge \sum_{i=1}^m r(F'_{j_i} \cap A)
         = \sum_{i=1}^m r(F_i \cap A)
         = r(\calF \cap A)
         = r(\calF).
       \]
  In the case of equality $ r(\calF') = r(\calF) $ we must have $ q=m $
  (\ie $ j_i = i $ for all $ i=1,\dots,m $) and $ r(F'_i) = r(F'_i \cap A) $
  for all $i$. But then $ F'_i $ and $ F_i $ are two flats of $M$ containing
  the set $ F'_i \cap A = F_i \cap A $, where
    \[ r(F'_i) = r(F'_i \cap A) = r(F_i \cap A) = r(F_i). \]
  This requires both $ F'_i $ and $ F_i $ to be the closure of $ F'_i \cap A =
  F_i \cap A $ for all $i$. In particular, we then have $ \calF' = \calF $,
  completing the proof of \ref{lem-rank-b}.
\end{proof}

\begin{example} \label{ex-rank}
  Let $M$ be the uniform rank-$3$ matroid on $ N=\{0,1,2,3\} $, so $
  B(M) = L^3_2 $ as in Example \ref{ex-matroid}. Then $ B(M) $ has cones of
  ranks $1$, $2$, and $3$, corresponding to the following chains of flats:
  \begin{enumerate}
  \item four $1$-dimensional cones of rank $1$ spanned by a unit vector,
    corresponding to the chains $ \calF = (\{i\}) $ for $ 0 \le i \le 3 $;
  \item six $1$-dimensional cones of rank $2$ spanned by a vector with two
    entries $1$ and two entries $0$, corresponding to the chains $ \calF =
    (\{i,j\}) $ for $ 0 \le i < j \le 3 $;
  \item twelve $2$-dimensional cones of rank $3$, corresponding to the chains $
    \calF = (\{i\},\{i,j\}) $ for $ 0 \le i,j \le 3 $ with $ i \neq j $.
  \end{enumerate}
  Let us apply (the proof of) Lemma \ref{lem-rank} to the first type of
  chain, say to $ \calF = (\{0\}) $ of rank $1$ with corresponding cone $
  \sigma_\calF $ spanned by $ [v_{\{0\}}] = [1,0,0,0] $. We extend $ \calF $ to
  a maximal chain of flats, \eg to $ \emptyset \subsetneq \{0\} \subsetneq
  \{0,1\} \subsetneq N $, and derive from this the basis $ A=\{0,1,2\} $ of
  $M$. Projecting $ B(M) $ with $ p^A $ (as in the picture in Construction
  \ref{constr-project}) we see indeed that $ \sigma_\calF $ is mapped
  injectively to its image cone spanned by $ [1,0,0] $, and that there are
  three more cones with the same image --- namely the ones corresponding to the
  chains $ (\{0,3\}) $, $ (\{0\},\{0,3\}) $, and $ (\{3\},\{0,3\}) $ --- and
  they all have bigger rank.
\end{example}

With this result we can now reconstruct arbitrary $1$-cycles in matroid fans
from their projections, by reconstructing their rays by descending induction on
the rank of the cone in which they lie.

\begin{corollary}[Reconstruction of $1$-cycles in matroid fans with
    projections] \label{cor-reconstruct}
  Let $ C,C' \in Z_1^{\trop}(B(M)) $ be two $1$-cycles contained in the Bergman
  fan $ B(M) $.

  If $ p^A_* (C) = p^A_* (C') $ for all bases $ A \subset N $ of $M$, then $ C
  = C' $.
\end{corollary}

\begin{proof}
  Let $ u \in \ZZ^N / \gen\one $ be a primitive vector contained in the support
  of $ B(M) $, let $ \sigma $ be the unique cone of $ B(M) $ containing $u$ in
  its relative interior, and let $ \calF $ be the corresponding chain of flats
  of $M$. Moreover, let $ \lambda $ and $ \lambda' $ be the multiplicities of
  $u$ in $C$ resp.\ $C'$. We have to prove that $ \lambda = \lambda' $.

  We will show this by descending induction on the rank $ r(\calF) $ as in
  Definition \ref{def-rank}. The start of the induction is trivial by Lemma
  \ref{lem-rank}, since the possible values of $ r(\calF) $ for a given
  matroid fan are bounded. For the induction step choose a basis $ A \subset N
  $ of $M$ for $ \calF $ as in Lemma \ref{lem-rank}, and let $ w \in \ZZ^A /
  \gen\one $ be the primitive vector pointing in the direction of $ p^A(u) $.
  Let $ u_1,\dots,u_m \in \ZZ^N / \gen\one $ be the primitive vectors occurring
  in $C$ or $C'$ except $u$ that are mapped by $ p^A $ to a positive multiple
  of $w$, and let $ \lambda_1,\dots,\lambda_m $ and $
  \lambda'_1,\dots,\lambda'_m $ be their multiplicities in $C$ resp.\ $C'$.
  Then the multiplicity of $w$ in $ p^A_* (C) = p^A_* (C') $ is
    \[ \lambda \, [\ZZ w: \ZZ p^A(u)]
       + \sum_{i=1}^m \lambda_i \, [\ZZ w: \ZZ p^A(u_i)]
       = \lambda' \, [\ZZ w: \ZZ p^A(u)]
       + \sum_{i=1}^m \lambda'_i \, [\ZZ w: \ZZ p^A(u_i)] \]
  by the definition of the push-forward of tropical cycles in Construction
  \ref{constr-push}. Now Lemma \ref{lem-rank} tells us that the vectors $
  u_1,\dots,u_m $ must lie in cones of rank bigger than $ r(\calF) $ (there can
  be no such vectors in $\calF$ by part \ref{lem-rank-a} of the lemma, and none
  in other cones of the same or smaller rank than $ r(\calF) $ by
  \ref{lem-rank-b}), and thus $ \lambda_i = \lambda'_i $ for all $ i=1,\dots,m
  $ by the induction assumption. Hence, we conclude by the above equation that $
  \lambda = \lambda' $.
\end{proof}

\begin{remark} \label{rem-reconstruct}
  Note that by Lemma \ref{lem-project} \ref{lem-project-e} the required
  coordinate projections to reconstruct $1$-cycles in $ B(M) $ are precisely
  those that map $ B(M) $ surjectively to a real vector space $ \RR^A /
  \gen\one $ of dimension $ r(N)-1 = \dim B(M) $. For example, in the case of $
  L^3_2 $ of Example \ref{ex-matroid}, corresponding to the uniform rank-$3$
  matroid on $ N=\{0,1,2,3\} $, we need the four coordinate projections
  $ p^A: \RR^N / \gen\one \to \RR^A / \gen\one \cong \RR^2 $ for all subsets
  $ A \subset N $ with $ |A|=3 $.
\end{remark}

\begin{remark}[Reconstruction in the non-constant coefficient case]
    \label{rem-nonconstant}
  In this paper we only consider the realizability by algebraic curves over a
  field with a trivial valuation. For curves over a non-Archimedean valued
  field, the so-called ``non-constant coefficient case'', one has to replace
  the $1$-dimensional fan cycles of Construction \ref{constr-intersection} by
  $1$-dimensional balanced polyhedral complexes modulo refinements as \eg in
  \cite[Section 1.1]{Rau09}. For these more general cycles there is a
  push-forward along a tropical morphism $f$ as well: one first has to choose a
  subdivision of the given cycle which is fine enough for the images of its
  cells under $f$ to form a polyhedral complex. Then one assigns to each such
  $1$-dimensional image cell $ \tau' $ the multiplicity $ \sum_{\tau: f(\tau) =
  \tau'} \lambda_\tau \, [\ZZ w : \ZZ f(u_\tau)] $, where $ \lambda_\tau $ is
  the multiplicity of $ \tau $, and $ u_\tau $ and $w$ are primitive integer
  vectors in the directions of $ \tau $ and $ \tau' $, respectively
  \cite[Section 1.3.2]{Rau09}. Comparing these multiplicities for two given
  cycles with the same arguments as in the proof above then shows that the
  statement of Corollary \ref{cor-reconstruct} also holds in this non-constant
  coefficient setting.
\end{remark}

To conclude this section, we will prove that the push-forwards occurring in
Construction \ref{cor-reconstruct} can easily be computed in terms of the sets
$ \P(C) $ of Notation \ref{notation-pofc}: To obtain $P(p^A_* C)$ one just has
to delete the coordinates not corresponding to $A$ of the elements in $P(C)$
and then add up all vectors that are positive multiples of each other.

\begin{lemma}[Projections of curves] \label{lem-pofc}
  Let $C$ be a tropical curve in $ B(M) $, and let $A$ be a basis of $M$. Then
  the set $ P(p^A_* C) $ consists exactly of the non-zero vectors of the form
    \[ \sum_{v \in \P(C): \, [p^A(v)] \in \sigma} p^A(v)
       \quad \in \ZZ^A \]
  for all $1$-dimensional cones $ \sigma $ in $ \RR^A/\gen\one $. In
  particular, we have $ \deg (p^A_* C) = \deg(C) $.
\end{lemma}

\begin{proof}
  We will show first that the vectors stated in the proposition satisfy the
  normalization requirement of Notation \ref{notation-pofc}, \ie that their
  minimal coordinate is $0$. Let $ \sigma $ be a $1$-dimensional cone in $
  \RR^A/\gen\one $ and $ v \in \P(C) $ with $ [p^A(v)] \in \sigma $. Moreover,
  let $ \calF $ and $ \calF' $ be the unique chains of flats of $M$ and $ M|_A
  $, respectively, such that $ [v] $ lies in the relative interior of the cone
  $ \sigma_\calF $ of $ B(M) $ and $ \sigma \backslash \{0\} $ lies in the
  relative interior of the cone $ \sigma_{\calF'} $ of $ B(M|_A) $. By Lemma
  \ref{lem-project} \ref{lem-project-b} we then have $ p^A(\sigma_\calF) =
  \sigma_{\calF'} $, and thus $ \calF' = \calF \cap A $. Now choose $ i \in A $
  not contained in any flat in $ \calF' $. Then $i$ cannot be contained in any
  flat $F$ of $ \calF $ either, because otherwise $ \calF' = \calF \cap A $
  implies $ A \subset F $ and hence $ r(F) \ge r(A) = r(N) $, yielding the
  contradiction $ F=N $. By Construction \ref{constr-matroid} this means that
  the $i$-th coordinate is minimal, and thus $0$, in $v$. As $i$ does not
  depend on $v$ but only on $ \sigma $, it follows that the minimal coordinate
  is $0$ in all the vectors stated in the lemma.

  Next, let us check that the multiplicity of each cone $ \sigma $ as above
  in $ p^A_* C $ is correct. As the minimal coordinate of each $ v \in \P(C) $
  with $ [p^A(v)] \in \sigma $ is $0$, by Construction \ref{constr-push} the
  cone $ \sigma $ receives a multiplicity of
    \[ m_C(v) \cdot [\RR p^A(u) \cap (\ZZ^A/\gen\one) : \ZZ p^A(u)]
       = \gcd \{ v_i: i \in N \} \cdot
         \frac {\gcd\{ v_i: i \in A \}}{\gcd\{ v_i: i \in N \}}
       = \gcd \{ v_i: i \in A \} \]
  from $v$, where $u$ denotes the primitive integral vector in $ \ZZ^N/\gen\one
  $ in the direction of $ [v] $. But as the minimal coordinate of $ p^A(v) $ is
  $0$ as well, this is exactly the integer length of $ [p^A(v)] $ in $
  \ZZ^A/\gen\one $. So the multiplicities of the rays are indeed correct.
  
  Finally, the statement about the degree now follows immediately from Lemma
  \ref{lem-degree}.
\end{proof}

%% file: pics/ex-matroid.tex
\begin{picture}(0,0)%
\includegraphics{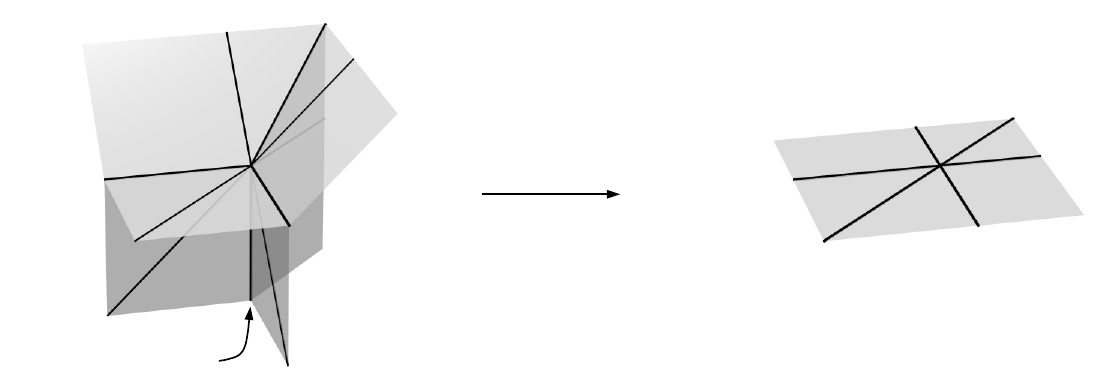}%
\end{picture}%
\setlength{\unitlength}{4144sp}%
\begingroup\makeatletter\ifx\SetFigFont\undefined%
\gdef\SetFigFont#1#2#3#4#5{%
  \reset@font\fontsize{#1}{#2pt}%
  \fontfamily{#3}\fontseries{#4}\fontshape{#5}%
  \selectfont}%
\fi\endgroup%
\begin{picture}(5085,1741)(450,-1774)
\put(2071,-1636){\makebox(0,0)[lb]{\smash{{\SetFigFont{11}{13.2}{\familydefault}{\mddefault}{\updefault}{\color[rgb]{0,0,0}$L^3_2$}%
}}}}
\put(5221,-1636){\makebox(0,0)[lb]{\smash{{\SetFigFont{11}{13.2}{\familydefault}{\mddefault}{\updefault}{\color[rgb]{0,0,0}$L^2_2$}%
}}}}
\put(1963,-168){\makebox(0,0)[lb]{\smash{{\SetFigFont{9}{10.8}{\familydefault}{\mddefault}{\updefault}{\color[rgb]{0,0,0}$[0,0,1,0]$}%
}}}}
\put(904,-894){\makebox(0,0)[rb]{\smash{{\SetFigFont{9}{10.8}{\familydefault}{\mddefault}{\updefault}{\color[rgb]{0,0,0}$[1,0,0,0]$}%
}}}}
\put(1029,-1164){\makebox(0,0)[rb]{\smash{{\SetFigFont{9}{10.8}{\familydefault}{\mddefault}{\updefault}{\color[rgb]{0,0,0}$[1,1,0,0]$}%
}}}}
\put(909,-1502){\makebox(0,0)[rb]{\smash{{\SetFigFont{9}{10.8}{\familydefault}{\mddefault}{\updefault}{\color[rgb]{0,0,0}$[1,0,0,1]$}%
}}}}
\put(1421,-1714){\makebox(0,0)[rb]{\smash{{\SetFigFont{9}{10.8}{\familydefault}{\mddefault}{\updefault}{\color[rgb]{0,0,0}$[0,0,0,1]$}%
}}}}
\put(1792,-1133){\makebox(0,0)[lb]{\smash{{\SetFigFont{9}{10.8}{\familydefault}{\mddefault}{\updefault}{\color[rgb]{0,0,0}$[0,1,0,0]$}%
}}}}
\put(4054,-894){\makebox(0,0)[rb]{\smash{{\SetFigFont{9}{10.8}{\familydefault}{\mddefault}{\updefault}{\color[rgb]{0,0,0}$[1,0,0]$}%
}}}}
\put(4179,-1164){\makebox(0,0)[rb]{\smash{{\SetFigFont{9}{10.8}{\familydefault}{\mddefault}{\updefault}{\color[rgb]{0,0,0}$[1,1,0]$}%
}}}}
\put(5113,-611){\makebox(0,0)[lb]{\smash{{\SetFigFont{9}{10.8}{\familydefault}{\mddefault}{\updefault}{\color[rgb]{0,0,0}$[0,0,1]$}%
}}}}
\put(4942,-1133){\makebox(0,0)[lb]{\smash{{\SetFigFont{9}{10.8}{\familydefault}{\mddefault}{\updefault}{\color[rgb]{0,0,0}$[0,1,0]$}%
}}}}
\put(2971,-1231){\makebox(0,0)[b]{\smash{{\SetFigFont{11}{13.2}{\familydefault}{\mddefault}{\updefault}{\color[rgb]{0,0,0}$A=\{0,1,2\}$}%
}}}}
\put(2971,-871){\makebox(0,0)[b]{\smash{{\SetFigFont{11}{13.2}{\familydefault}{\mddefault}{\updefault}{\color[rgb]{0,0,0}$p^A$}%
}}}}
\end{picture}%

%% file: realize.tex
\section{Relative Realizability} \label{sec-realize}

\subsection{Tropicalization} \label{subsec-trop}

As explained in the introduction, we are concerned with a particular tropical
inverse problem, \ie a question of which tropical cycles are realizable as
tropicalizations of algebraic varieties satisfying certain given conditions. To
describe our relative realization problem, we will use the following setup.

\begin{notation}\label{basicnotation}
  Throughout this section, let $K$ be an algebraically closed field, and fix
  $n\in \N$. We denote by $R=K[x_0,\ldots,x_n]$ the polynomial ring, and by
  $S=K[x_0^{\pm 1},\ldots,x_n^{\pm 1}]$ the Laurent polynomial ring in $n+1$
  variables over $K$. Consider the standard grading on $R$ and $S$, and let
  $S_0=K[\frac{x_i}{x_j}: 0\leq i,j\leq n]$ be the $K$-algebra of elements of
  degree $0$ in $S$. Let $X=\Spec(S_0)$ be the $n$-dimensional torus over $K$.
  Algebraic varieties and curves are always assumed to be irreducible.
\end{notation}

\begin{construction}[Tropicalization of varieties]\label{constr-trop}
  Let $X$ be the torus as in Notation \ref{basicnotation}. We define a
  tropicalization of subvarieties of $X$ to $\R^{n+1}/\left\langle
  \bf{1}\right\rangle$ as follows. Let $f=\sum_{\nu} a_{\nu}x^{\nu}\in S_0$
  with $a_{\nu}\in K$ (for $\nu\in \Z^{n+1}$ such that $\sum_i \nu_i=0$). For
  $\omega\in \R^{n+1}$ let $c_\omega=\min_{\nu:\,a_\nu \neq 0} \omega \cdot \nu
  $, where we use the standard scalar product on $\R^{n+1}$. Then we set
  $\inom_{\omega}(f)=\sum_{\nu:\, \omega\cdot \nu=c_\omega} a_{\nu} x^{\nu}$ to
  be the initial polynomial with respect to $\omega $, and for an ideal
  $I\subset S_0$ we set $\inom_{\omega}(I)=(\inom_{\omega}(f):f\in I)$.

  For a subvariety $Y\subset X$ given by a prime ideal $P\subset S_0$ we
  consider the set $\trop(Y)=\left\{\omega\in \R^{n+1}: \inom_{\omega}(P)\neq
  S_0\right\}$. This is the underlying set of a polyhedral fan in $\R^{n+1}$
  induced by the Gr\"obner fan of $PS\cap R$. For every maximal cone
  $\sigma$ in this fan we define a multiplicity
    $$ \mult(\sigma)=\sum_{\inom_{\omega}(P)\subset Q}
       \ell((S_0/\inom_{\omega}(P))_Q), $$
  where $\omega$ is any element of the relative interior of $\sigma$, the sum
  is taken over all minimal primes $Q$ of $\inom_{\omega}(P)$, and $\ell(M)$
  denotes the length of the $S_0$-module $M$. By \cite{SP} we know that
  $\trop(Y)$ is a balanced polyhedral fan with these multiplicities, and thus a
  tropical variety in the sense of Notation \ref{not-cycles}.

  As $\left\langle \bf{1}\right\rangle$ is contained in the lineality space of
  $\trop(Y)$ it is in fact more natural to consider $\trop(Y)$ in
  $\R^{n+1}/\left\langle \bf{1}\right\rangle$. Hence, from now on, by the
  \emph{tropicalization} $ \trop(Y) $ of $Y$ we will always mean this balanced
  polyhedral fan in $\R^{n+1}/\left\langle \bf{1}\right\rangle$. We have $\dim
  Y = \dim \, \trop(Y)$ after this identification by \cite[Theorem 2.1.2]{SP}
  together with the first main result of \cite{BG}.
\end{construction}

\begin{example}[Linear varieties] \label{ex-trop-linear}
  Let $L$ be a linear ideal in $S$, \ie an ideal that can be generated by $ n-k
  $ independent linear forms $ l_1,\dots,l_{n-k} $ in $x_0,\ldots,x_n$, where
  $ k+1 = \dim(S/L) $. With $ L_0 = L \cap S_0 $ we then have $ \dim (S_0/L_0)
  = k $. We will call $ E = \Spec(S_0/L_0) $ a $k$-dimensional \emph{linear
  variety} in the torus $ X = \Spec (S_0) $.

  The tropicalization of $E$ is easy to describe: let $Q$ be a $
  (k+1)\times (n+1) $ matrix whose rows span the zero set of $(l_1,\dots,
  l_{n-k}) $, and denote by $ M(L) $ the matroid on the columns of $Q$ in the
  sense of \cite[Proposition 1.1.1]{Oxl}. Then $\trop(E)$ is exactly the
  associated matroid fan $B(M(L))$ as in Construction \ref{constr-matroid},
  see \cite[Theorem 1]{AK}. For example, if $L$ and thus also $Q$ is general,
  we obtain for $ M(L) $ the uniform matroid of rank $k+1$ on $n+1$ elements,
  and consequently $ \trop(E) = B(M(L)) = L^n_k $ as in Example
  \ref{ex-matroid}.
\end{example}

In the case of curves, let us extend the tropicalization map from varieties
to cycles.

\begin{definition}[Tropicalization of $1$-cycles] \label{def-tropcycles}
  Let $E$ be a subvariety of the torus $X$.
  \begin{enumerate}
  \item \label{def-tropcycles-a}
    We denote by $ Z_1(E) $ the group of $1$-cycles in $E$ in the sense of
    \cite{F}, \ie the free abelian group generated by all curves in $E$.
    Moreover, let $ Z_1^+(E) \subset Z_1(E) $ be the subset of effective
    cycles. By abuse of notation, for a curve $ Y \subset E $ we will also
    write $ Y \in Z_1^+(E) $ for the $1$-cycle whose multiplicity is $1$ on $Y$
    and $0$ on all other curves.
  \item \label{def-tropcycles-b}
    For $ Y \in Z_1(E) $ we define the \emph{degree} $ \deg(Y) \in \ZZ $ of $Y$
    to be the intersection product of $Y$ with a general hyperplane in $X$.
    The subsets of $ Z_1(E) $ and $ Z_1^+(E) $ of cycles of degree $d$ will be
    denoted $ (Z_1(E))_d $ resp.\ $ (Z_1^+(E))_d $.
  \item \label{def-tropcycles-c}
    We extend the tropicalization of Construction \ref{constr-trop} by
    linearity to a group homomorphism
      \[ \qquad\quad
         \Trop: Z_1(E) \longrightarrow Z_1^{\trop}(\trop(E)). \]
  \end{enumerate} 
\end{definition}

\begin{lemma}[Tropicalization preserves the degree] \label{lem-tropdeg}
  For any subvariety $ E \subset X $ and cycle $ Y \in Z_1(E) $ we have $ \deg
  (\Trop(Y)) = \deg(Y) $, with the tropical degree as in Definition
  \ref{def-degree}.
\end{lemma}

\begin{proof}
  By linearity it suffices to prove the statement for a curve $ Y \subset E $.
  Let $G$ be a general hyperplane in $X$, so that $ \trop(G) = L^n_{n-1} $ by
  Example \ref{ex-trop-linear}. Moreover, let $ \Delta $ be a complete
  unimodular fan in $ \RR^{n+1}/\gen\one $ containing $ \trop(Y) $ and $
  \trop(G) $ as subfans. Then the corresponding toric variety $ X(\Delta) $
  is complete and smooth. Hence, by \cite[Theorem 3.1]{FS} there is a natural
  ring homomorphism $ \phi: A_*(X(\Delta)) \to Z^{\trop}_*(\RR^{n+1}/\gen\one)
  $, where $A_*(X(\Delta))$ denotes the Chow homology of $X(\Delta)$, and the
  ring structures are given by the algebraic resp.\ tropical intersection
  product. By the fundamental theorem of tropical geometry \cite{SP} together
  with \cite[Corollary 3.15]{ST} we know that $\phi([\overline W])=\trop(W)$
  for every $ W \subset X $ such that $\trop(W) $ is a subfan of $ \Delta $. In
  particular,
    $$ \phi([\overline{G}]\cdot[\overline{Y}])
       = \phi([\overline{G}]) \cdot \phi([\overline{Y}])
       = L^n_{n-1} \cdot \trop(Y) = \deg(\trop(Y)), $$
  where we identify $ Z_0^{\trop}(\RR^{n+1}/\gen\one) $ with $ \ZZ $. But $
  \overline G $ and $ \overline Y $ only intersect in the torus $X$ of
  $X(\Delta)$, since $G$ is general. So we have $ \deg([\overline{G}] \cdot
  [\overline{Y}]) = \deg(Y) $, and the result follows.
\end{proof}

In this paper we will always consider curves or $1$-cycles in a fixed
$2$-dimensional linear variety in $X$. So let us fix the following notation.

\begin{notation}[Planes] \label{not-plane}
  In the following, let $ L=(l_1,\dots,l_{n-2}) $ always be a linear ideal in
  $S$ with $ \dim (S/L) = 3 $ as in Example \ref{ex-trop-linear}, and let
  $ L_0 = L \cap S_0 $. Then $ E := \Spec(S_0/L_0) $ is a \emph{plane} in the
  torus $X$, and $ \trop(E) = B(M(L)) $ will be called a \emph{tropical plane}.
  By abuse of notation, the ideal $ L \cap R $ in $R$ will also be denoted by
  $L$.
\end{notation}

\begin{construction}[$ (Z_1^+(E))_d $ as an algebraic variety]
    \label{constr-moduli}
  We now want to give the sets $ (Z_1^+(E))_d $ of effective $1$-cycles of
  degree $d$ in the plane $E$ the structure of an algebraic variety.

  By \cite[Proposition 6.11]{Har}, every effective $1$-cycle $Y$ in $E$ is the
  divisor in $E$ of a regular function on $E$, \ie of an element of $ S_0/L_0
  $, unique up to units. Hence, $ Z_1^+(E) $ is in natural bijection with the
  set of principal ideals in $ S_0/L_0 $. As $ S_0/L_0 $ is just the degree-$0$
  part of $ S/L $, this set corresponds by extension to the set of principal
  homogeneous ideals in $ S/L $, which in turn by localization corresponds to
  the set of principal ideals in $ R/L $ generated by a homogeneous polynomial
  without monomial factors \cite[Proposition 2.2]{E}. Let $ f^Y \in R/L $ be
  a homogeneous polynomial such that the ideal $ (f^Y) \subset R/L $
  corresponds to $ Y \in Z_1^+(E) $ in this way. Note that the choice of $ f^Y
  $ is unique up to multiplication with an element of $ K^* $.

  Geometrically, the plane $E$ is a dense open subset of the projective space $
  \Proj(R/L) \cong \PP^2 $. So by taking the closure, a cycle $ Y \in Z_1^+(E)
  $ determines a cycle $ \overline Y \in Z_1^+(\Proj(R/L)) $ without components
  in coordinate hyperplanes, which is just the divisor of $ f^Y $. In
  particular, by B\'ezout's theorem we see that $ \deg f^Y = \deg Y $. So we
  get an injective map
    \[ (Z_1^+(E))_d \hooklongrightarrow \PP((R/L)_d), \;\; Y \mapsto [f^Y] \]
  to a projective space of dimension $ \frac{d(d+3)}2 $, where $ (R/L)_d $
  denotes the degree-$d$ part of $ R/L $. As explained above, its image is the
  complement of the space of polynomials without monomial factors, and thus a
  dense open subset. As $\PP((R/L)_d)$ is irreducible, this gives $
  (Z_1^+(E))_d $ the structure of an open subvariety of $ \PP((R/L)_d) $. In
  the following we will always consider $ (Z_1^+(E))_d $ as an algebraic
  variety in this way.
\end{construction}

It is now the goal of this paper to study which tropical curves in $ \trop(E) $
can be realized as tropicalizations of effective $1$-cycles in $E$, and to
describe the space of such cycles.

\begin{defn}[Relative realizability and realization space] \label{def-real}
  Let $E$ be a plane in the torus $X$ defined by a linear ideal $L$ as in
  Notation \ref{not-plane}. Moreover, let $C\in Z_1^{\trop}(\trop(E))$ be a
  tropical curve of degree $d$ as in Notation \ref{not-cycles} and Definition
  \ref{def-degree}.
  \begin{enumerate}
  \item \label{def-real-a}
    We say that $C$ is \emph{(relatively) realizable} in $E$ (or in $L$) if
    there exists an effective cycle $Y\in Z_1^+(E)$ with $ \Trop(Y)=C $ (note
    that we must have $ \deg(Y)=d $ in this case by Lemma \ref{lem-tropdeg}).
  \item \label{def-real-b}
    The subset $ \real(C) \subset (Z_1^+(E))_d $ of all effective cycles $Y$ such that
    $ \Trop(Y)=C $ is called the \emph{(relative) realization space} of $C$. We
    will see in Algorithm \ref{algorithm} that $ \real(C) $ is the complement
    of a union of hyperplanes in a linear space; its dimension will be called
    the \emph{realization dimension} $\realdim(C) $ of $C$.
  \end{enumerate}
\end{defn}

\subsection{Projecting to the plane}\label{projplane}

As realizability is completely understood in the plane case (see Lemma
\ref{lem-trop-plane}), one idea to deal with more complicated inverse problems
is to reduce them to questions about this case. In this section we will give an
equivalent description of our realization problem in terms of several dependent
realization problems in the plane, which can then be attacked algorithmically.

To relate our problem to the plane case, we will use coordinate projections
onto the plane which preserve enough information both on the algebraic and the
tropical side, such that from all images of those projections we can
reconstruct the original objects. By Corollary \ref{cor-reconstruct} these can
be chosen to be all projections onto coordinates indexed by the bases of the
matroid $M(L)$. On the algebraic side such a projection is simply a
monomorphism and can geometrically be imagined as pushing the plane $E$ in
space injectively to a coordinate plane. It can be described as follows.

\begin{construction}[Algebraic projections to the plane] \label{constr-algproj}
  Let $A=\left\{j_0,j_1,j_2\right\}\subset \left\{0,\ldots,n\right\}$ be a
  basis of $M(L)$. Let $ R^A = K[x_{j_0},x_{j_1},x_{j_2}] $ and $
  S^{A}_0=K[\frac{x_i}{x_k}:i,k=j_0,j_1,j_2]$, and consider the $K$-algebra
  monomorphism $ S^{A}_0\to S_0/L_0$ mapping $\frac{x_i}{x_k}$ to $
  \frac{x_i}{x_k}+L_0 $. This defines an injective morphism
  $\pi^A$ of affine varieties from $ E = \Spec(S_0/L_0) $ to the
  $2$-dimensional torus $ E^A := \Spec(S^{A}_0) $, and thus an injective group
  homomorphism
    \[ \pi^A_*: Z_1(E) \hooklongrightarrow Z_1(\pi^A(E)) \hooklongrightarrow
       Z_1(E^A) \]
  of the corresponding cycle groups. Note that $ \pi^A_* $ preserves degrees
  and effective cycles, and thus maps $ (Z_1^+(E))_d \subset \PP((R/L)_d) $ to
  $ (Z_1^+(E^A))_d \subset \PP(R^A_d) $. To describe $ \pi^A_* $ explicitly in
  terms of these ambient spaces, note that $A$ being a basis of $M(L) $ implies
  that the $K$-algebra homomorphism of polynomial rings
    \[ R^A \to R/L, \;\; x_i \mapsto \overline{x_i} \]
  inducing the map $ S^A_0 \to S_0/L_0 $ from above is an isomorphism. There is
  thus an inverse morphism
    \[ R/L \to R^A, \;\;
       \overline{x_i} \mapsto a_{i,j_0}x_{j_0} + a_{i,j_1}x_{j_1}
       + a_{i,j_2}x_{j_2} \]
  with unique $ a_{i,j} \in K $ such that $ x_i - a_{i,j_0}x_{j_0} -
  a_{i,j_1}x_{j_1} - a_{i,j_2}x_{j_2} \in L $ for all $i$. It maps the class of
  a polynomial $ f \in R $ to the polynomial $ f_A \in R^A $ obtained from $f$
  by replacing $ x_i $ with $ a_{i,j_0}x_{j_0} + a_{i,j_1}x_{j_1} +
  a_{i,j_2}x_{j_2} $ for all $i$, \ie by eliminating all $ x_i $ with $ i
  \notin A $ modulo $L$. The map $ \pi^A_*: (Z_1^+(E))_d \to (Z_1^+(E^A))_d $
  is thus obtained by restriction from the isomorphism
    \[ \PP((R/L)_d) \to \PP(R^A_d), \;\; [\overline f] \mapsto [f_A]. \]
\end{construction}

The main idea of our algorithm deciding the relative realizability problem
relies on the fact that projection commutes with tropicalization. More
precisely, we have the following commutative diagram described in \cite[Theorem
1.1]{ST}:

\begin{theorem}[Sturmfels-Tevelev]\label{tevelev}
  For every basis $A$ of $M(L)$ the diagram

  \def\message#commdiag{}%
  {\begin{center}\input{pics/commdiag}\end{center}}

  commutes, where $ p^A $ is the projection as in Construction
  \ref{constr-project}.
\end{theorem}

Note that all maps in this diagram preserve the degree of cycles by Lemma
\ref{lem-pofc} and Lemma \ref{lem-tropdeg}. Moreover, note that $
p^A_*(\trop(E)) $ is simply the plane $ \RR^A/\gen\one \cong \RR^2 $. We can
thus reduce our relative realizability problem to a finite number of dependent
realizability problems in the plane case, as stated in the following theorem.

\begin{theorem}[Tropicalization by projections] \label{mainthm}
  Let $C\subset \trop(E)$ be a $1$-dimensional tropical cycle, and let $Y \in
  Z_1(E)$. Then the following are equivalent:
  \begin{enumerate}
  \item \label{mainthm-a}
    $\Trop(Y)=C$,
  \item \label{mainthm-b}
    $\Trop(\pi^A_*(Y))=p^{A}_*(C)$ for every basis $A$ of $M(L)$.
  \end{enumerate}
\end{theorem}

\begin{proof}
  By Theorem \ref{tevelev} it is clear that \ref{mainthm-a} $ \Rightarrow $
  \ref{mainthm-b}. We show \ref{mainthm-b} $ \Rightarrow $ \ref{mainthm-a}: For
  every basis $A$ of $M(L)$ we have $\Trop(\pi^A_*(Y))=p^{A}_{*}(\Trop(Y))$ by
  Theorem \ref{tevelev}. Hence, $p^{A}_{*}(\Trop(Y))=p^{A}_*(C)$ for every
  basis $A$ for the two $1$-cycles $\Trop(Y)$ and $C$, which are both contained
  in $\trop(E)$. By Corollary \ref{cor-reconstruct} this implies $\Trop(Y)=C$.
\end{proof}

\subsection{The plane case: Newton polytopes} \label{subsec-newton}

To study the realization problem in the plane $ \RR^3/\gen\one $ we can use
Newton polytopes. To describe this setup, let us consider the
case $ n=2 $, \ie $ R = K[x_0,x_1,x_2] $, $ S_0 = K[(\frac{x_1}{x_0})^{\pm 1},
(\frac{x_2}{x_0})^{\pm 1}] $, $ E=X=\Spec(S_0) $ is the $2$-dimensional torus,
and thus $ \trop(E) = \RR^3/\gen\one $. Moreover, let $ H = \{ v \in \R^3:
v_0+v_1+v_2=0 \} $ be the vector space dual to $ \RR^3/\gen\one $, with dual
lattice $ \{ v \in \ZZ^3: v_0+v_1+v_2=0 \} $. For $ d \in \NN $ we set
$ H_d = \{ v \in \R^3: v_0+v_1+v_2=d \} $ and $ \Delta_d = \conv \{
(d,0,0), (0,d,0), (0,0,d) \} \subset H_d $.

\begin{construction}[Inner normal fans] \label{constr-normfan}
  Let $P$ be a lattice polytope in $H_d $ for some $d\in\NN$ with $\dim P\geq
  1$. We consider the \emph{inner normal fan} $N(P)\subset \R^3/\left\langle
  {\bf 1}\right\rangle$ of $P$ in the sense of \cite[Example 7.3]{Zie},
  possibly refined such that the origin is a cone. For each $1$-dimensional
  cone $\sigma$ in $N(P)$ we define the multiplicity $m(\sigma)$ to be the
  lattice length of the corresponding edge of $P$. Then $ N(P) $ is a tropical
  curve in $ \RR^3/\gen\one $, and for a given $d$ the assignment $ P \mapsto
  N(P) $ yields a one-to-one correspondence between lattice polytopes in $H_d
  $ of dimension at least $1$ up to translation and plane tropical curves
  \cite[Corollary 2.5]{MI}.
\end{construction}

\begin{definition}[Newton polytopes of polynomials, $1$-cycles, and tropical
    curves] \label{def-newton} \leavevmode \vspace{-0.5ex}
  \begin{enumerate}
  \item \label{def-newton-a}
    The \emph{Newton polytope} of a homogeneous polynomial $ f = \sum_{\nu}
    a_{\nu}x^{\nu} \in R $ of degree $d$ is defined to be $ \New(f) = \conv
    \{\nu: a_{\nu}\neq 0 \} \subset \Delta_d \subset H_d $.
  \item \label{def-newton-b}
    The \emph{Newton polytope} $ \New(Y) $ of an effective $1$-cycle $ Y \in
    Z_1^+(E) $ of degree $d$ is defined to be the Newton polytope of a
    polynomial $ f \in R_d $ corresponding to $Y$ via the inclusion $
    (Z_1^+(E))_d \subset \PP(R_d) $ of Construction \ref{constr-moduli}.
  \item \label{def-newton-c}
    The \emph{Newton polytope} $ \New(C) $ of a tropical curve $C$ in $
    \RR^3/\gen\one $ is the unique lattice polytope in $ \RR^3 $ such that
    \begin{itemize}
    \item the inner normal fan of $ \New(C) $ is $C$ (this fixes the polytope
      up to translation by Construction \ref{constr-normfan}),
    \item $ \New(C) \subset \Delta_d $ and meets all three sides of $ \Delta_d
      $ for some $d$.
    \end{itemize}
  \end{enumerate}
\end{definition}

\begin{lemma}[Tropicalization of plane $1$-cycles] \label{lem-trop-plane}
  Let $ Y \in Z_1^+(E) $ and $ C \in Z_1^{\trop}(\RR^3/\gen\one) $. Then
  $ \Trop(Y) = C $ if and only if $ \New(Y) = \New(C) $.

  In particular, every tropical curve in $ \RR^3/\gen\one $ is realizable, and
  the number $d$ in Definition \ref{def-newton} \ref{def-newton-c} is the
  degree of $C$.
\end{lemma}

\begin{proof}
  By \cite[Corollary 2.1.2]{EKL06} the tropicalization of $Y$ is just the inner
  normal fan of $ \New(Y) $. So if $ \New(Y)=\New(C) $ then $ \Trop(Y) $ equals
  the inner normal fan of $ \New(C) $, which is $C$. Conversely, if $
  \Trop(Y)=C $, the inner normal fan of $ \New(Y) $ is $C$. Moreover, this
  polytope is the Newton polytope of a homogeneous polynomial without monomial
  factors by Construction \ref{constr-moduli}, so it is contained in $ \Delta_d
  $ with $ d = \deg(Y) $ and meets all three sides of $ \Delta_d $. Hence, $
  \New(Y) = \New(C) $ by Definition \ref{def-newton} \ref{def-newton-c}.

  In particular, if $C$ is a tropical curve in $ \RR^3/\gen\one $ with $
  \New(C) \subset \Delta_d $ we can choose a homogeneous polynomial of degree
  $d$ with Newton polytope $ \New(C) $. As this polynomial then does not have a
  monomial factor, it determines a cycle $ Y \in (Z_1^+(E))_d $ with $ \New(Y)
  = \New(C) $, and thus with $ \Trop(Y) = C $. So $C$ is realizable, and by
  Lemma \ref{lem-tropdeg} we have $ d = \deg(Y) = \deg(C) $.
\end{proof}

Note that this result gives us explicit conditions on an effective $1$-cycle
realizing a tropical curve. This will play an important role in our algorithm.

\subsection{Computing Realizability} \label{subsec-comp}

Let us now collect our results to obtain an algorithm to detect realizability
and compute the realization space and dimension of a curve $C$ in a tropical
plane $ \trop(E) $.

Note that $C$ can only be realized by effective cycles of degree $ d=\deg(C) $
by Lemma \ref{lem-tropdeg}. So to deal with our problem algorithmically we
first of all need to choose coordinates on the space $ (Z_1^+(E))_d $ of such
cycles. This is easily achieved, since we have $ (Z_1^+(E))_d \subset
\PP((R/L)_d) $ as an open subset by Construction \ref{constr-moduli}, and $
\PP((R/L)_d) \cong \PP(R^B_d) $ for a basis $B$ of the matroid $ M(L) $ and $
R=K[x_i: i \in B] $ by Construction \ref{constr-algproj}. So we can choose
homogeneous coordinates for the projective space $ \PP(R^B_d) $, \ie the
coefficients of a homogeneous polynomial of degree $d$ in three variables $ x_i
$ with $ i \in B $, as homogeneous coordinates for $ (Z_1^+(E))_d $. The
resulting algorithm to detect realizability and compute the realization space
and dimension can be described as follows.

\begin{alg}[Realizability of curves in a tropical plane] \label{algorithm}
  Consider a plane $E\subset X$ given by a linear ideal $L\subset S$ with
  $\dim(S/L)=3$, and let $C$ be a tropical curve in $\trop(E)$.
  \begin{enumerate}
  \item Compute the degree of $C$: by Lemma \ref{lem-degree} this is just the
    natural number $d$ such that $\sum_{v \in \P(C)} v = d \cdot {\bf 1} $.
  \item Compute a basis $B=\left\{j_0,j_1,j_2\right\}$ of the matroid $M(L)$
    associated to $L$.
  \item Let
      $$ \qquad\quad
         f=\sum_{\nu\in \N^{3}, |\nu|=d} a_{\nu}x^{\nu}
         \in K[x_{j_0},x_{j_1},x_{j_2}], $$
    where the $a_{\nu}$ are parameters in $K$ that form the coordinates of the
    projective space $ \PP(R^B_d) $ containing our moduli space $ (Z_1^+(E))_d
    $ as explained above. More precisely, we can consider $f$ to be in the
    polynomial ring
      $$ \qquad\quad
         K [a_{\nu}:\nu\in \N^{3},|\nu|=d][x_{j_0},x_{j_1},x_{j_2}]. $$
  \item For every basis $A$ of $M(L)$ compute the polynomial $ f_A \in K[x_i: i
    \in A] $ as in Construction \ref{constr-algproj} by eliminating all $ x_i $
    with $ i \notin A $ from $f$ modulo $L$. Note that the coefficients of
    $f_A$ are linear forms in the $a_{\nu}$ determined by $L$.
  \item On the other hand, for every basis $A$ compute the tropical
    push-forward $ C_A=p^A_*(C) $ by Lemma \ref{lem-pofc}, and its Newton
    polytope $\New(C_A)$ as in Definition \ref{def-newton} \ref{def-newton-c}.
    This can \eg be done explicitly by Lemma \ref{lemma:newton}. Note that $
    \deg(C_A) = d $ for all $A$ by Lemma \ref{lem-pofc}, and thus $ \New(C_A)
    \subset \Delta_d $ by Lemma \ref{lem-trop-plane}.
  \item Obtain conditions on the $a_{\nu}$ to ensure that $ \New(f_A) =
    \New(C_A) $ for all bases $A$: if $ f_A=\sum_{\nu}b_{\nu}x^{\nu} $, this
    means that $b_{\nu}=0$ if $\nu\notin \New(C_A)$, whereas $b_{\nu}\neq 0$
    if $\nu$ is a vertex of $\New(C_A)$. On the other coefficients of $f_A$
    there are no conditions. This gives a set of linear equalities and
    non-equalities in the $a_{\nu}$.
  \item Let $ R(C) \subset \PP(R^B_d) \cong \PP((R/L)_d) $ be the solution set
    of these linear equalities and non-equalities. Then
    \begin{itemize}
    \item $ R(C) \subset (Z_1^+(E))_d $, \ie no polynomial in the solution set
      contains a monomial factor: if $ \overline x_i \,|\, \overline f $ in $
      R/L $ for some $ i \in \{0,\dots,n\} $ and $[\overline f] \in R(C)$
      then $ x_i \,|\, f_A $ for every basis $A$ with $ i \in A $. Hence,
      $ \New(f_A) $ does not meet all three sides of $ \Delta_d $ whereas $
      \New(C_A) $ does --- in contradiction to $ [f] \in R(C) $. So every $ [f]
      \in R(C) $ corresponds to an effective $1$-cycle $Y$ of degree $d$.
    \item $ R(C) $ describes exactly the effective $1$-cycles tropicalizing to
      $C$: this follows from Theorem \ref{mainthm}, since $ \New(f_A) =
      \New(C_A) $ means $ \New(\pi_*^A (Y)) = \New(p_*^A(C)) $ by Construction
      \ref{constr-algproj}, which in turn means $ \Trop(\pi_*^A (Y)) = p_*^A(C)
      $ by Lemma \ref{lem-trop-plane}.
    \end{itemize}
    Hence, $ R(C) $ is the realization space $ \real(C) $ of $C$.
  \end{enumerate}
  In particular, $ \real(C) \subset (Z_1^+(E))_d \subset \PP((R/L)_d) $ is the
  complement of a union of hyperplanes in a linear space. Of course, the
  algorithm can also be used to compute the dimension of $ \real(C) $.
\end{alg}

For some purposes, it is interesting to know if a tropical curve in $\trop(E)$
is relatively realizable not only by a positive cycle, but also by an
irreducible curve in $E$. Here, a cycle $Y \in Z_1(E)$ is called
\emph{irreducible} if it is defined by exactly one irreducible curve in $E$
with multiplicity one. Otherwise, we call this cycle \emph{reducible}.

To check for this irreducible realizability of a tropical curve $ C\subset
\trop(E)$, we will consider non-trivial decompositions of $C$ into $D_1$ and
$D_2$, by which we mean two positive tropical cycles $D_1,D_2\neq 0$ in
$Z_1^{\trop}(\trop(E))$ such that $D_1+D_2=C$. Note that in this case
$\deg(C)=\deg(D_1)+\deg(D_2)$.

\begin{proposition} \label{prop-irreal}
  Let $C\subset \trop(E)$ be a tropical curve with $\realdim(C)=m$ as in 
  Definition \ref{def-real}. Then there is an irreducible cycle in $E$
  tropicalizing to $C$ if and only if for every non-trivial decomposition
  $C=D_1+D_2$ of $C$ in $Z_1^{\trop}(\trop(E))$ with $\realdim(D_1)=m_1$ and
  $\realdim(D_2)=m_2$ we have $m_1+m_2<m$.
\end{proposition}

\begin{proof}
  Let $C=D_1+D_2$ be a non-trivial decomposition with $\deg(D_1)=d_1$ and
  $\deg(D_2)=d_2$, where $ d_1 + d_2 = d = \deg(C) $. We consider the morphism
  of varieties
  \begin{align*}
    \phi: \mathbb{P}((R/L)_{d_1})\times \mathbb{P}((R/L)_{d_2})
      & \longrightarrow \mathbb{P}((R/L)_d)\\
    ([\overline{f}],[\overline{g}]) & \longmapsto [\overline{fg}]. 
  \end{align*}
  As every polynomial of degree $d$ has only finitely many factorizations into
  two polynomials of degree $d_1$ and $d_2$ up to scalar multiplication, the
  map $\phi$ has finite fibers. Hence, the dimension of the image of $
  \real(D_1)\times\real(D_2)$, \ie the space of cycles in $ \real(C) $ that
  are composed of two cycles of degrees $d_1$ and $d_2$, has dimension
  $m_1+m_2$. Note that $\phi(\real(D_1)\times\real(D_2))$ is closed in
  $\real(C)$.

  If $m=m_1+m_2$ this implies that $\real(C)=\phi(\real(D_1)\times\real(D_2))$,
  since $ \real(C) $ is irreducible by Algorithm \ref{algorithm} as an open
  subset of a linear space. So in this case $C$ is only realizable by a sum of
  two cycles realizing $D_1$ and $D_2$, respectively.

  On the other hand, if $m_1+m_2<m$ for every non-trivial decomposition of $C$
  into $D_1+D_2$, the space of reducible cycles in $\real(C)$ is a proper
  closed subset of $\real(C)$. Hence, there has to be an irreducible curve in
  $\real(C)$, so in this case $C$ is irreducibly realizable.
\end{proof}

Using this, Algorithm \ref{algorithm} can be extended to check the irreducible
relative realizability of a tropical curve in $\trop(E)$:

\begin{alg} \label{algorithm-irr}
  Let $C$ be a realizable tropical curve in $\trop(E)$ and $m$ the realization
  dimension of $C$ in $E$. For any non-trivial decomposition $C = D_1 + D_2 \in
  Z_1^{\trop}(\trop(E))$, where $D_1$ and $D_2$ are two tropical curves in
  $\trop(E)$, compute the realization dimensions $m_i$ of $D_i$ in $E$ for $ i
  = 1,2$. If there is a decomposition $C = D_1 + D_2$ with $m_1 + m_2 = m$,
  then the tropical curve $C$ is not realizable by an irreducible curve in $E$;
  if $m_1 + m_2 < m$ for all decompositions $C = D_1 + D_2$, then there is an
  irreducible curve in $E$ tropicalizing to $C$.
\end{alg}

\begin{Example}[The Singular library \emph{realizationMatroids.lib}]
    \label{ex-singular}
  The above Algorithms \ref{algorithm} and \ref{algorithm-irr} deciding
  (irreducible) relative realizability are implemented in the Singular library
  \emph{realizationMatroids.lib} \cite{Sing,Win12}.

  Let $C$ be a tropical curve in $ \trop(E) = B(M(L))$, and let $ P=\P(C) $
  as in Notation \ref{notation-pofc}. The functions
  \emph{realizationDim}$(L,P)$ and \emph{irrRealizationDim}$(L,P)$ then either
  return the (irreducible) realization dimension of $C$ in $E$, or $-1$ if the
  realization space of $C$ is empty, \ie if $C$ is not relatively realizable
  (by an irreducible curve).

  Moreover, following Algorithm \ref{algorithm} we can explicitly describe the
  set of polynomials $f$ such that the ideal $L+(f)$ tropicalizes to a given
  tropical curve $C$ in $\trop(E)$. Correspondingly, the function
  \emph{realizationDimPoly}$(L,P)$ returns the realization dimension of $C$
  together with a polynomial realizing $C$ in $E$. This function works by
  checking small integer coefficients for the polynomials and may only be used
  if the characteristic of $K$ is zero. 

  The following example shows these functions for the curve $C$ in $ L^3_2 =
  \trop(V(L)) $ with $ L=(x_0+x_1+x_2+x_3) \subset K[x_0,\dots,x_3] $ and
  $ \P(C)=\{(2,2,0,0),(0,0,2,2)\} $.

  \begin{verbatim}
  > LIB "realizationMatroids.lib";
  > ring r = 0,(x0,x1,x2,x3),dp;
  > ideal L = x0+x1+x2+x3;
  > list P = list(intvec(2,2,0,0),intvec(0,0,2,2));
  > realizationDim(L,P);
  0
  > irrRealizationDim(L,P);
  -1
  > realizationDimPoly(L,P);
  0 x0^2+2*x0*x1+x1^2
  \end{verbatim}
\end{Example}

%% file: pics/commdiag.tex
\begin{picture}(0,0)%
\includegraphics{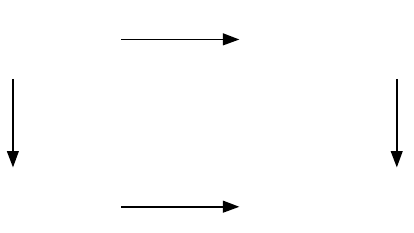}%
\end{picture}%
\setlength{\unitlength}{4144sp}%
\begingroup\makeatletter\ifx\SetFigFont\undefined%
\gdef\SetFigFont#1#2#3#4#5{%
  \reset@font\fontsize{#1}{#2pt}%
  \fontfamily{#3}\fontseries{#4}\fontshape{#5}%
  \selectfont}%
\fi\endgroup%
\begin{picture}(1875,1063)(1291,-1304)
\put(1351,-466){\makebox(0,0)[b]{\smash{{\SetFigFont{11}{13.2}{\familydefault}{\mddefault}{\updefault}{\color[rgb]{0,0,0}$Z_1(E)$}%
}}}}
\put(3106,-466){\makebox(0,0)[b]{\smash{{\SetFigFont{11}{13.2}{\familydefault}{\mddefault}{\updefault}{\color[rgb]{0,0,0}$Z_1^{\trop}(\trop(E))$}%
}}}}
\put(3106,-1231){\makebox(0,0)[b]{\smash{{\SetFigFont{11}{13.2}{\familydefault}{\mddefault}{\updefault}{\color[rgb]{0,0,0}$Z_1^{\trop}(p^A_* \trop(E))$}%
}}}}
\put(2071,-376){\makebox(0,0)[b]{\smash{{\SetFigFont{9}{10.8}{\familydefault}{\mddefault}{\updefault}{\color[rgb]{0,0,0}$\Trop$}%
}}}}
\put(2071,-1141){\makebox(0,0)[b]{\smash{{\SetFigFont{9}{10.8}{\familydefault}{\mddefault}{\updefault}{\color[rgb]{0,0,0}$\Trop$}%
}}}}
\put(3151,-803){\makebox(0,0)[lb]{\smash{{\SetFigFont{9}{10.8}{\familydefault}{\mddefault}{\updefault}{\color[rgb]{0,0,0}$p^A_*$}%
}}}}
\put(1351,-1231){\makebox(0,0)[b]{\smash{{\SetFigFont{11}{13.2}{\familydefault}{\mddefault}{\updefault}{\color[rgb]{0,0,0}$Z_1(\pi^A(E))$}%
}}}}
\put(1306,-826){\makebox(0,0)[rb]{\smash{{\SetFigFont{9}{10.8}{\familydefault}{\mddefault}{\updefault}{\color[rgb]{0,0,0}$\pi^A_*$}%
}}}}
\end{picture}%

%% file: criteria.tex
\section{General Criteria for Relative Realizability} \label{sec-criteria}

\subsection{Relative Realizability of Cycles}

As in Notation \ref{not-plane} let $B(M(L))=\trop(E) \subset
\R^{n+1}/\left\langle \bf{1}\right\rangle$ be the matroid fan obtained by
tropicalizing a plane $E$ in a torus $X$. In this section we will show that any
tropical cycle in $B(M(L))$ is relatively realizable by a cycle in $Z_1(E)$,
\ie that the map $\Trop: Z_1(E) \to Z_1^{\trop}(\trop(E))$ of Definition
\ref{def-tropcycles} \ref{def-tropcycles-c} is surjective. In particular, this
means that any tropical curve in $B(M(L))$ is relatively realizable by a cycle
in $Z_1(E)$. To prove this claim, we start by showing that any tropical curve
in $B(M(L))$ containing at most one $1$-dimensional cone not corresponding
to a rank-1 flat is relatively realizable.

\vspace{-1.5ex}

\begin{sidepic}{prop-0} \begin{Lemma}\label{lem:positiveReal}
  Let $C$ be a tropical curve in $\trop(E)$ such that $\P(C) =
  \{v,\lambda_1 v_{F_1},\dots,\lambda_r v_{F_r}\}$, where $ [v] \in \trop(E)
  \setminus \{0\}$, $\lambda_i \in \mathbb N$ for all $i$, and $F_1,\ldots,F_r$
  are rank-1 flats of the matroid $ M(L) $ with associated vectors $
  v_{F_1},\dots,v_{F_r} $ as in Construction \ref{constr-matroid}. Then the
  tropical curve $C$ is realizable in $L$.
\end{Lemma} \end{sidepic}

\vspace{-1.5ex}

\begin{proof}
  Since $\trop(E)$ consists of the cones $\sigma_{\mathcal G}$, where $\mathcal
  G = (G_1,G_2)$ is a chain of flats in $M(L)$, and $ [v] \in \trop(E)$, there
  is a rank-1 flat $G_1$ and a rank-2 flat $G_2$ such that $ [v] \in
  \sigma_{\mathcal G} = \cone([v_{G_1}],[v_{G_2}])$. We denote the elements of
  the base set of $M = M(L)$ by $0,\ldots,n$. Applying a coordinate
  permutation, we may assume that there are $0 < k_1 < k_2 \leq n$ 
  such that $G_1 = \{0,\ldots,k_1-1\}$ and $G_2 =\{0,\ldots,k_2-1\}$. Hence,
  we have
    $$ v = a (e_0+\ldots+e_{k_1-1}) + b (e_{k_1} + \ldots + e_{k_2-1}) $$
  for some $a,b \in \mathbb N$, $ b \le a \le d := \deg(C) $, and $(a,b) \neq
  (0,0)$. With
    $$ f = c_0 x_0^d + c_1 x_{k_1}^d + c_2 x_0^b x_{k_2}^{d-b}
         + c_3 x_{k_1}^a x_{k_2}^{d-a} $$
  for any generic $ c = (c_0,\dots,c_3) \in K^4 $, we claim that $C =
  \Trop(L+(f))$. To prove this claim, we want to use Theorem \ref{mainthm} and
  therefore show that $\Newt(f_A) = \Newt(p_*^A (C))$ for all bases $A$ of $M$,
  where $f_A$ is as in Construction \ref{constr-algproj}.
  
  So let $ A=\{i,j,k\} $ a basis of $M$. By Construction \ref{constr-algproj}
  there are $ \lambda_l, \mu_l, \nu_l \in K$ for $l=i,j,k$ such that
  \begin{align*}
    x_0 - \lambda_i x_i - \lambda_j x_j - \lambda_k x_k &\quad \in L, \\
    x_{k_1} - \mu_i x_i - \mu_j x_j - \mu_k x_k &\quad \in L, \\
    x_{k_2} - \nu_i x_i - \nu_j x_j - \nu_k x_k &\quad \in L.
  \end{align*}
  Since any two different elements in $\{0,\ldots,k_1-1\}$ are contained in
  the rank-1 flat $G_1$, they are linearly dependent, as are three pairwise
  different elements in $ G_2 = \{0,\ldots,k_2-1\}$. Hence, we get the
  following 3 cases.

  \textbf{Case 1:} $i \in \{0,\ldots,k_1-1\}, j \in \{k_1,\ldots,k_2-1\}, k \in
  \{k_2,\ldots,n\}$.

  As $\{0,i\}$ and $\{k_1,i,j\}$ are linearly dependent, we have $\lambda_j
  = \lambda_k = \mu_k = 0$. It also holds that $\lambda_i, \mu_j, \nu_k \neq 0$,
  because $L$ does not contain a monomial and $\{0,k_1,k_2\}$ is a basis of
  $M$. So we have
  \begin{align*}
    f_{A}(x_i,x_j,x_k)
      &= c_0 (\lambda_i x_i)^d + c_1 (\mu_i x_i + \mu_j x_j)^d
         + c_2 (\lambda_i x_i)^b (\nu_i x_i + \nu_j x_j + \nu_k x_k)^{d-b} \\
      &\qquad + c_3 (\mu_i x_i + \mu_j x_j)^a
         (\nu_i x_i + \nu_j x_j + \nu_k x_k)^{d-a}.
  \end{align*}
  Since $\lambda_i, \mu_j, \nu_k \neq 0$ and $ c \in K^4 $ is generic, the
  coefficients of $x_i^d$, $x_j^d$, $x_i^b x_k^{d-b}$, and $ x_j^a x_k^{d-a} $
  in this polynomial are non-zero. However, the coefficient of $
  x_i^{m_i}x_j^{m_j}x_k^{m_k}$ is zero whenever $m_i < b$ and $m_k > d-a$.
  As $\P(p_*^{A}(C)) = \{(a,b,0),(d-a,0,0),(0,d-b,0),(0,0,d)\}$, we
  thus have
    $$ \Newt(f_{A}) = \conv((b,0,d-b),(d,0,0),(0,d,0),(0,a,d-a))
       = \Newt(p_*^{A}(C)). $$

  \textbf{Case 2:} $i, j \in \{k_1,\ldots,k_2-1\}, k \in \{k_2,\ldots,n\}$.

  In this case $\{0,i,j\}$ and $\{k_1,i,j\}$ are linearly dependent, and so
  we have $\lambda_k = \mu_k = 0$. Moreover, we have $\lambda_i, \lambda_j,
  \mu_i, \mu_j, \nu_k \neq 0$ as $ \{0,j\}, \{0,i\}, \{k_1,j\}, \{k_1,i\}$
  resp.\ $\{k_2,i,j\}$ are linearly independent. Since 
  \begin{align*}
    & f_{A}(x_i,x_j,x_k) = \\
    &\quad c_0 (\lambda_i x_i + \lambda_j x_j)^d
       + c_1 (\mu_i x_i + \mu_j x_j)^d
       + c_2 (\lambda_i x_i + \lambda_j x_j)^b
             (\nu_i x_i + \nu_j x_j + \nu_k x_k)^{d-b} \\
    &\quad c_3 (\mu_i x_i + \mu_j x_j)^a
             (\nu_i x_i + \nu_j x_j + \nu_k x_k)^{d-a}
  \end{align*}
  and $c \in K^4$ is generic, we have as in Case 1
    $$ \Newt(f_{A}) = \conv((b,0,d-b),(d,0,0),(0,d,0),(0,b,d-b))
       = \Newt(p_*^{A}(C)). $$

  \textbf{Case 3:} $j, k \in \{k_2,\ldots,n\}$.

  We have $(\lambda_l,\mu_l) \neq (0,0)$ for all $l=i,j,k$: if for instance
  $\lambda_i = \mu_i = 0$, then we see by suitable linear combinations of the
  polynomials $ x_0 - \lambda_j x_j - \lambda_k x_k \in L$ and $x_{k_1} - \mu_j
  x_j - \mu_k x_k \in L$ that $\{0,k_1,k\}$ and $\{0,k_1,j\}$ are linearly
  dependent. 

  Since $c \in K^4$ is generic, the coefficient of $x_i^d$ in $f_{A}$,
  \ie the sum $c_0 \lambda_i^d + c_1 \mu_i^d + c_2 \lambda_i^b \nu_i^{d-b} +
  c_3 \mu_i^a \nu_i^{d-a}$, is non-zero, as is the coefficient of $x_j^d$ and
  $x_k^d$. Hence, we have
    $$ \Newt(f_{A}) = \conv( (d,0,0),(0,d,0),(0,0,d) )
       = \Newt(p_*^{A} (C)). $$
  So with generic coefficients $c \in K^4$, we have $\Newt(p_*^A (C)) =
  \Newt(f_A)$ for all bases $A$ of $M$. Applying Theorem \ref{mainthm}, we see
  that $\Trop(L+(f)) = C$.
\end{proof}

\begin{Example} \label{ex:real}
  Let $F_1,\ldots,F_k$ be the rank-1 flats of $M(L)$. Note that $\sum_{i=1}^k
  v_{F_i} = \bf{1} $, \ie these $1$--dimensional cones (with multiplicities
  $1$) form a balanced polyhedral fan $ D $ in $ \trop(E) $ with $ P(D) =
  \{v_{F_1},\ldots,v_{F_k}\} $. By Lemma \ref{lem:positiveReal}, this tropical
  curve is relatively realizable in $L$.
\end{Example}

Using Lemma \ref{lem:positiveReal}, we can now prove the following
proposition.

\begin{Proposition} \label{prop-realize-cycle}
  The map $\Trop: Z_1(E) \to Z_1^{\trop}(\trop(E))$ is surjective.
\end{Proposition}

\begin{proof}
  Let $ C \in Z_1^{\trop}(\trop(E)) $ be a tropical curve and $v \in \P(C)$
  not a positive multiple of one of the vectors $ v_{F_1},\dots,v_{F_k} $
  corresponding to the rank-$1$ flats $ F_1,\dots,F_k $ of $M(L)$. As in the
  proof of Lemma \ref{lem:positiveReal} we know that $[v] \in
  \cone\{[v_{G_1}],[v_{G_2}]\}$, where $ (G_1,G_2) $ is a chain of flats in
  $M(L)$. Due to the definition of $v_{G_2}$ we know that $v_{G_2} = \sum_{G
  \subset G_2} v_G$, with the sum taken over all rank-$1$ flats $G$ contained
  in $ G_2 $. So we can write $ [v] = \sum_{i=1}^k a_i [v_{F_i}] $ for some $
  a_1,\dots,a_k \in \NN $. With $ d = \max \{ a_1,\dots,a_k \} + 1 $ there is
  then a (balanced) tropical curve $ D_v $ of degree $d$ with $ \P(D_v) = \{ v,
  (d-a_1) \, v_{F_1},\dots,(d-a_k) \, v_{F_k} \} $.

  By construction, $ C-\sum_v D_v $ is now a balanced cycle in $ \trop(E) $
  with rays $ [v_{F_1}],\dots,[v_{F_k}] $, where the sum is taken over all $v$
  as above. As $ v_{F_1},\dots,v_{F_k} $ are linearly independent, this is only
  possible if $ C-\sum_v D_v = \lambda \, D $ is a multiple of the tropical
  curve $D$ of Example \ref{ex:real} with $ \P(D) = \{v_{F_1},\dots,v_{F_k}\}
  $. But $ D_v $ and $D$ are realizable by Lemma \ref{lem:positiveReal} and
  Example \ref{ex:real}. Thus we have $ C = \lambda \, D + \sum_v D_v \in
  \Trop(Z_1(E))$. The map $\Trop: Z_1(E) \to Z_1^{\trop}(\trop(E))$ is linear
  and moreover, any cycle in $Z_1^{\trop}(\trop(E))$ is $\mathbb Z$--linear
  combination of tropical curves. This shows the surjectivity of $\Trop$.
\end{proof}

\subsection{Obstructions to Realizability in
  \texorpdfstring{$L = (x_0+x_1+x_2+x_3)$}{L=(x0+x1+x2+x3)}}

In this section, we will use Algorithm \ref{algorithm} to give general
obstructions to realizability in $L = (x_0+x_1 +x_2 +x_3) \subset K[x_0^{\pm
1},x_1^{\pm 1},x_2^{\pm 1},x_3^{\pm 1}]$, where $K$ is any algebraically closed
field. More precisely, we will work out conditions on a tropical curve $C$
implying that the Newton polytopes of the tropical push-forwards $p^A_*C$ of
$C$ cannot be the Newton polytopes of the algebraic projections $f_A$ of any
homogeneous polynomial $f$ whose degree is the degree of the tropical curve
$C$. In Theorem \ref{mainthm}, we have seen that in this case, the tropical
curve $C$ is not realizable in $L$. Some of these criteria will depend on the
characteristic of $K$, others are independent of the characteristic. Our main
results are the Propositions \ref{prop:intprod}, \ref{prop:commonray},
\ref{prop:oppositefaces}, \ref{prop:oneside}, and \ref{prop:bogartkatz}.

To use dependencies between the Newton polytopes of the push-forwards
$p_*^A(C)$ of a tropical curve $C$ in $L_2^3$, we will need the coordinates of
the vertices of these Newton polytopes. Therefore, we will use an explicit
version of the definition of the Newton polytope of a plane tropical curve
summarized in the following lemma.

\begin{remark}[Orientations of $ \RR^3/\gen\one $] \label{rem-orientation}
  For our computation of Newton polytopes it is convenient to choose an
  orientation of the plane $ \RR^3/\gen\one $ by calling a basis $
  ([v_1],[v_2]) $ positive if $ \det (v_1,v_2,\one) > 0 $. For the set $ \P(C)
  = \{ v_1,\dots,v_r \} $ of a tropical curve $C$ in $ \RR^3/\gen\one $ we will
  assume that its vectors $ v_1,\dots,v_r \in \RR^3 $ are listed
  in \emph{positive order}, \ie that their classes $ [v_1],\dots,[v_r] $ are in
  positive order with respect to the above orientation, ending with a positive
  multiple of $ e_0 $ (if present). This means that these vectors $ v_i =
  (v_{i,0},v_{i,1},v_{i,2}) $ are sorted like the columns in the following
  table:
    \[ \stackrel {\text{ascending $i$}}{\longrightarrow} \qquad
       \begin{array}{|c|cccccc|} \hline
         v_{i,0} & a & 0 & 0 & 0 & b & a \\
         v_{i,1} & b & a & a & 0 & 0 & 0 \\
         v_{i,2} & 0 & 0 & b & a & a & 0 \\ \hline
       \end{array} \]
  where $ a,b \in \NN_{>0} $ are arbitrary numbers (depending on $i$), and
  vectors that belong to the same column are sorted according to ascending
  values of $ \frac ba $.
\end{remark}

\begin{Lemma}[Vertices of Newton polytopes] \label{lemma:newton}
  Let $C$ be a plane tropical curve of degree $d$ with $\P(C) = \{v_1,\dots,
  v_r\}$, where the vectors $ v_i = (v_{i,0},v_{i,1},v_{i,2}) $ for $ i=1,
  \dots,k $ are sorted in positive order as in Remark \ref{rem-orientation}.
  Then the vertices of $ \Newt(C) \subset \Delta_d $ are exactly the points
    \[ Q_k = (0, m, d-m) + \sum_{i=1}^k u_i
       \quad\text{with}\quad
       m = \sum_{i:\,v_{i,1} \neq 0} v_{i,0}
       \;\;\text{and}\;\;
       u_i = (v_{i,1}-v_{i,2},v_{i,2}-v_{i,0},v_{i,0}-v_{i,1}) \]
  for $ k=0,\dots,r-1 $.
\end{Lemma}

\begin{proof}
  As the vectors $ [v_1],\dots,[v_r] \in \RR^3/\gen\one $ sum up to zero and
  are sorted in positive order, the points $ \sum_{i=1}^k [v_i] $ for $
  k=0,\dots,r-1 $ are the vertices of a convex polytope. But $ Q_0,\dots,
  Q_{r-1} $ are by construction just the images of these points under an affine
  isomorphism $ \RR^3/\gen\one \to H_d := \{(w_0,w_1,w_2): w_0+w_1+w_2=d \} $.
  Thus they also form the vertices of a convex polytope $ \Delta \subset H_d $.

  To check that $C$ is the weighted inner normal fan of $ \Delta $ note first
  of all that one of the coordinates of each $ v_k $ is zero by definition of $
  \P(C) $. Hence, for all $k$ we have $ \gcd(v_{k,0}, v_{k,1}, v_{k,2}) =
  \gcd(u_{k,0}, u_{k,1}, u_{k,2}) $, and thus $ [v_k] $ and $ u_k = Q_k -
  Q_{k-1} $ have the same integer length. Moreover, it is obvious that $ v_k
  \cdot u_k = 0 $. This means that the function $ \varphi_k: \Delta \to \RR, \;
  u \mapsto v_k \cdot u $ is constant, and thus extremal, on the side $
  \overline{Q_{k-1}Q_k} $ of $ \Delta $. It remains to be shown that it is in
  fact minimal there. This is obvious if $ \Delta $ is $1$-dimensional, so let
  us assume that $ \Delta $ is $2$-dimensional. In this case balancing requires
  the oriented angle between $ v_k $ and $ v_{k+1} $ (where we set $ v_{r+1} :=
  v_1 $) to be in the open interval $ (0,\pi) $, which means that $ ([v_k],
  [v_{k+1}]) $ is a positive basis of $ \RR^3/\gen\one $. But then we have
    \[ \varphi_k(Q_{k+1}) - \varphi_k(Q_k) = v_k \cdot u_{k+1}
         = \det (v_k,v_{k+1},\one) > 0, \]
  which implies minimality as claimed. So, up to translations, $ \Delta $ is
  the Newton polytope of $C$.

  Finally, to see that the translation is correct, it suffices by Definition
  \ref{def-newton} to check that $ \Delta $ meets the three lines in $ H_d $
  where one of the coordinates is $0$. Obviously, the $0$-th coordinate is $0$
  for $ Q_0 $. If $k$ corresponds to the last vector in the second column in
  the table of Remark \ref{rem-orientation}, the first coordinate of $ Q_k $ is
    \[ m + \sum_{i=1}^k (v_{i,2}-v_{i,0})
       = \sum_{i:\,v_{i,1} \neq 0} v_{i,0} - \sum_{i:\,v_{i,1} \neq 0} v_{i,0}
       = 0, \]
  and if $k$ corresponds to the last vector in the fourth column in this table,
  the second coordinate of $ Q_k $ is
    \[ d-m + \sum_{i=1}^k (v_{i,0}-v_{i,1})
       = d - \sum_{i:\,v_{i,1} \neq 0} v_{i,0}
           + \sum_{i:\,v_{i,1} \neq 0} v_{i,0}
           - \sum_{i=1}^r v_{i,1}
       = d-m+m-d = 0 \]
  by Lemma \ref{lem-degree}.
\end{proof}

\begin{Notations}[Notations for Newton polytopes] \label{notations:newton}
  \leavevmode \vspace{-0.5ex}
  \begin{enumerate}
  \item
    To simplify the notations, we denote the projections $ p^A: \RR^4/\gen
    \one \to \RR^3/\gen\one $ for a basis $A$ of $M(L)$ (see Construction
    \ref{constr-project}) by $ p^k $, where $ k \in \{0,1,2,3\} $ is the
    unique element not contained in $A$. Correspondingly, the polynomial $ f_A
    $ of Construction \ref{constr-algproj} is denoted by $f_k$.
  \item
    To work with the Newton polytopes of the push-forwards $p_*^k (C)$, we will
    identify the plane $ H_d = \{ w \in \mathbb R^3: w_0+w_1+w_2 = d \} $ with
    $\mathbb R^2$ by choosing the isomorphism $ H_d \to \mathbb R^2,\;w \mapsto
    (w_0,w_1)$. In other words, from now on the Newton polytopes $\Newt(p_*^k
    (C))$ and $\Newt(f_k)$ will be considered to be in $ \conv
    ((0,0),(0,d),(d,0)) \subset \mathbb R^2 $ by dropping the last coordinate
    which is not the $k$-th.
  \item \label{notations:newton-c}
    We will denote the coefficients of $f_0$ and $ f_3 $ by $a_{ij}$ and
    $b_{ij}$, respectively, \ie
      \[ \qquad\quad
         f_0(x_1,x_2,x_3)
         = \sum_{i=0}^d \sum_{j=0}^{d-i} a_{ij} x_1^i x_2^j x_3^{d-i-j}
         \quad\text{and}\quad
         f_3(x_0,x_1,x_2)
         = \sum_{i=0}^d \sum_{j=0}^{d-i} b_{ij} x_0^i x_1^j x_2^{d-i-j}. \]
    This is illustrated in the following picture:
    \begin{center} \begin{tikzpicture}
      \draw (0,0) -- (0,3) -- (3,0) -- (0,0);
      \filldraw (0,0) circle (2pt)
          (1,0) circle (2pt)
          (2,0) circle (2pt)
          (3,0) circle (2pt)
          (0,1) circle (2pt)
          (1,1) circle (2pt)
          (1,2) circle (2pt)
          (0,2) circle (2pt)
          (2,1) circle (2pt)
          (0,3) circle (2pt);
      \draw (0,0) node[anchor = north east] {$ \scriptstyle a_{0,0}$}
          (0,0) node[anchor = south west] {$ x_3^d $}
          (1,0) node[anchor = north] {$ \scriptstyle a_{1,0}$}
          (2,-0.1) node[anchor = north] {$\ldots$}
          (3,0) node[anchor = north] {$ \scriptstyle a_{d,0}$}
          (3,0) node[anchor = south west] {$ x_1^d $}
          (0,1) node[anchor = east] {$ \scriptstyle a_{0,1}$}
          (-0.2,2) node[anchor = east] {$\vdots$}
          (0,3) node[anchor = east] {$ \scriptstyle a_{0,d}$}
          (0,3) node[anchor = south west] {$ x_2^d $}
          (2.5,2.5) node {$\Newt(f_0)$};
      \draw (6,0) -- (6,3) -- (9,0) -- (6,0);
      \filldraw (6,0) circle (2pt)
          (7,0) circle (2pt)
          (8,0) circle (2pt)
          (9,0) circle (2pt)
          (6,1) circle (2pt)
          (7,1) circle (2pt)
          (7,2) circle (2pt)
          (6,2) circle (2pt)
          (8,1) circle (2pt)
          (6,3) circle (2pt);
      \draw (6,0) node[anchor = north east] {$ \scriptstyle b_{0,0}$}
          (6,0) node[anchor = south west] {$ x_2^d $}
          (7,0) node[anchor = north] {$ \scriptstyle b_{1,0}$}
          (8,-0.1) node[anchor = north] {$\ldots$}
          (9,0) node[anchor = north] {$ \scriptstyle b_{d,0}$}
          (9,0) node[anchor = south west] {$ x_0^d $}
          (6,1) node[anchor = east] {$ \scriptstyle b_{0,1}$}
          (5.8,2) node[anchor = east] {$\vdots$}
          (6,3) node[anchor = east] {$ \scriptstyle b_{0,d}$}
          (6,3) node[anchor = south west] {$ x_1^d $}
          (8.5,2.5) node {$\Newt(f_3)$};
    \end{tikzpicture} \end{center}
  \end{enumerate}
\end{Notations}

To start with, we want to prove an obstruction to realizability in $L$
equivalent to an obstruction given by Brugall\'e and Shaw in
\cite{brugalleshaw}. They proved that if a tropical curve $D$ in $L_2^3$ is
realizable in $L$ by an irreducible curve and $C$ is a tropical curve in
$L_2^3$ such that the tropical intersection product $C \cdot D$ in $L_2^3$ is
negative, then $C$ cannot be realizable by an irreducible curve. In the
following, we will give a result which is equivalent to this obstruction in the
case where $D$ is one of the classical lines $D_1 = \Span\{[1,1,0,0]\}$, $D_2 =
\Span\{[1,0,1,0]\}$, or $D_3 = \Span\{ [1,0,0,1] \}$. However, Brugall\'e and
Shaw always ask for relative realizability by irreducible cycles, while in this
paper we allow any positive cycle to realize a given tropical curve. That is
why the statements seem to be different at first glance, although in fact the
obstruction we obtain is the same.

\begin{Remark}[Correspondence between relative realizability by irreducible and
    positive cycles]
  In Proposition \ref{prop-irreal} we have seen that relative realizability by
  positive cycles may be used to decide relative realizability by irreducible
  cycles. However, it is also possible to decide relative realizability by
  positive cycles using the irreducible version of relative realizability: Let
  $C$ be any tropical curve in $L_2^3$. To decide whether or not $C$ is
  relatively realizable by a positive cycle in $Z_1(E)$, consider all positive
  tropical decompositions $C = \sum_{i=1}^r C_i \in Z_1^{\trop}(\trop(E))$. If
  there is a decomposition $C = \sum_{i=1}^r C_i $ such that each tropical
  curve $C_i$ is relatively realizable by an irreducible cycle, then $C$ is
  relatively realizable by a positive cycle. If such a decomposition does not
  exist, there is no positive cycle in $Z_1(E)$ tropicalizing to $C$.
\end{Remark}

Our algorithm to decide relative realizability by positive cycles is based on
projections, and a priori it is not clear how the intersection product of two
tropical curves in $L_2^3$ can be seen in these projections. Therefore, we
want to find an interpretation of the intersection product $C \cdot D$ in terms
of the Newton polytopes of the push-forwards $p_*^k(C)$ of $C$ in case $D$ is a
classical line. Since the three classical lines are equivalent modulo
permutations of the coordinates, we will only state and prove this
interpretation for the classical line $D_1$.

\begin{sidepic}{intprod}
  \begin{Lemma}[Intersection products with classical lines]
      \label{lemma:interpretation}
    Let $C$ be a tropical curve of degree $d$ in $L_2^3$. Moreover, let $m_3$
    be the maximum of all $m \in \mathbb N$ such that $(i,j) \notin \Newt(p_*^3
    (C))$ for all $(i,j) \in \mathbb N^2$ with $i+j < m$. Similarly, let $m_0$
    be the maximum of all $m \in \mathbb N$ such that $(i,j) \notin \Newt(p_*^0
    (C))$ for all $(i,j) \in \mathbb N^2$ with $i > d - m$. Then the tropical
    intersection product of $C$ and $D_1 = \Span\{[1,1,0,0]\}$ can be written
    as
      \[ C \cdot D_1 = d - m_0 - m_3. \]
  \end{Lemma}
\end{sidepic}

\begin{proof}
  In Construction \ref{constr-l32} we have seen that the tropical intersection
  product of $C$ and $D_1$ is given by
    \[ C \cdot D_1 = d - \sum\limits_{(a,b,0,0) \in \P(C)} \min \{a,b\}
                       - \sum\limits_{(0,0,a,b) \in \P(C)} \min \{a,b\}. \]
  To prove the claim, we will show that
    $$ \sum\limits_{(a,b,0,0) \in \P(C)} \min \{a,b\} = m_3. $$ 
  Analogously, one can then show that
    $$ \sum\limits_{(0,0,a,b) \in \P(C)} \min \{a,b\} = m_0. $$
  Let $(a_1,b_1,0,0),\ldots,(a_r,b_r,0,0)$ be all elements in $\P(C)$ whose
  first and second coordinate are both non-zero, sorted in such a way that 
    \[ \frac{b_k}{a_k} < \frac{b_{k+1}}{a_{k+1}}
       \quad\text{for all}\quad k = 1,\ldots,r-1. \]
  By Remark \ref{rem-orientation} this corresponds to the first vectors of a
  positive ordering of the rays of $p_*^3 (C)$. Hence, by Lemma
  \ref{lemma:newton} the point $ Q_k = \left(\sum_{l=1}^k b_l, \sum_{l=k+1}^r
  a_l\right) $ is a vertex of $\Newt(p_*^3(C))$ for all $k=0,\ldots,r$. We
  then have
    \[ \Newt(p_*^3(C)) \subset \conv\bigl( \{Q_k: k=0,\ldots,r \}
       \cup \{ (0,d),(d,0) \} \bigr). \]
  Since $ \sum\limits_{l=1}^k b_l+\sum\limits_{l=k+1}^r a_l \geq
  \sum\limits_{l=1}^r \min\{a_l,b_l\} $ for every $k$, we see that $ i+j \geq
  \sum\limits_{l=1}^r \min\{a_l,b_l\}$ for all $ (i,j) \in \Newt(p_*^3 (C)) $.
  On the other hand, the points are in such an ordering that there is a unique
  $0 \leq k \leq r$ with $ \frac {b_l}{a_l} \le 1$ for all $l \le k$ and $
  \frac{b_l}{a_l} > 1$ for all $l > k$. In other words, $b_l \le a_l$ for all
  $l \le k$ and $b_l > a_l$ for all $l > k$. So for this particular $k$ we get
    \[ \sum_{l=1}^k b_l + \sum_{l=k+1}^r a_l
       = \sum_{l=1}^k \min\{a_l,b_l\} + \sum_{l=k+1}^r \min\{a_l,b_l\}
       = \sum_{l=1}^r \min\{a_l,b_l\}. \]
  Hence, we have $m_3 = \sum\limits_{l=1}^r \min\{a_l,b_l\}$. 
\end{proof}

Using this interpretation of the intersection product in terms of the Newton
polytopes of the push-forwards $p_*^3 (C)$ and $p_*^0 (C)$, we are aiming to
prove the first obstruction to relative realizability, which will imply the
obstruction of Brugall\'e and Shaw in the special case when the realizable
curve is the classical line $ D_1 = \Span\{[1,1,0,0]\} $. To do so, we need the
following lemma.

\vspace{-1ex}

\begin{sidepic}{kristin} \begin{Lemma} \label{lemma:intprod}
  Let $C$ be a tropical curve of degree $d$ in $L_2^3$, and let $m_0$ and $m_3$
  be as in Lemma \ref{lemma:interpretation}. Moreover, set 
  \begin{align*}
    n_0 &= | \{ (d-m_0,j) \in \mathbb N^2:
      j \leq m_0, (m_0,j) \notin \Newt(p_*^0 (C)) \}| \quad \text{and} \\
    n_3 &= | \{ (i,j) \in \mathbb N^2:
      i + j = m_3, (i,j) \notin \Newt(p_*^3 (C))\}|.
  \end{align*}
  If $n_0 > d-m_3$ or $n_3 > d-m_0$, then $C$ is not relatively realizable in
  $L$.
\end{Lemma} \end{sidepic}

\begin{proof}
  It suffices by symmetry to show that $C$ is not relatively realizable if $
  n_0 > d-m_3 $. So assume to the contrary that this inequality holds and $ C =
  \Trop(L+(f)) $ for a homogeneous polynomial $ f \in K[x_0,x_1,x_2,x_3] $ of
  degree $d$.

  By definition of $ m_0 $ the polynomial $ f_0 $ of Notation
  \ref{notations:newton} \ref{notations:newton-c} (whose Newton polytope is
  shown in the picture above) has $ x_1 $-degree $ d-m_0 $. Hence, the
  (inhomogeneous) polynomial
    \[ g(x_3) := \frac{\partial^{d-m_0} f_0}{\partial x_1^{d-m_0}} (x_1,1,x_3)
       \]
  is independent of $ x_1 $; its coefficients correspond to the vertical
  dotted line in the picture. By definition of $ n_0 $ there are thus $ 0 \le s
  \le t \le m_0 $ with $ t-s = m_0 - n_0 $ such that $g$ contains only
  terms of degrees from $s$ to $t$, with non-zero coefficients in degrees
  $s$ and $t$. Hence, $g$ has degree $t$ and a zero of order $s$ at $0$.

  On the other hand, the polynomial $ f_3(x_0,x_1,x_2) =
  f_0(x_1,x_2,-x_0-x_1-x_2) $ is of $ x_2 $-degree $ d-m_3 $ by definition of $
  m_3 $. Substituting $ -x_3-x_1 $ for $ x_0 $ in $ f_3 $, we thus get another
  polynomial $ f_0(x_1,x_2,x_3-x_2) $ whose $ x_2 $-degree is at most $ d-m_3
  $. At the same time this polynomial has $ x_1 $-degree at most $ d-m_0 $
  just as $ f_0(x_1,x_2,x_3) $. Consequently, it does not contain any term with
  a power of $ x_3 $ less than $ m_0+m_3-d $, \ie $ x_3^{m_0+m_3-d} $ divides $
  f_0(x_1,x_2,x_3-x_2) $. Setting $ x_2=1 $ and replacing $ x_3 $ by $ x_3+1 $
  we thus see that $ (x_3+1)^{m_0+m_3-d} $ divides $ f_0(x_1,1,x_3) $, and
  hence, also $g$. In other words, $g$ has a zero of order at least $ m_0+m_3-d
  $ at $ -1 $.

  Altogether, we have now seen that the polynomial $g$ has degree $t$ but
  zeroes of total order
    \[ s + (m_0+m_3-d) = (t+n_0-m_0)+m_0+m_3-d > (t+d-m_3-m_0)+m_0+m_3-d
       = t, \]
  which is a contradiction.
\end{proof}

\begin{Proposition} \label{prop:intprod}
  Let $C$ be a tropical curve in $ L^3_2 $, and let $D$ be the classical line
  $D = \Span\{[1,1,0,0]\}$. Let $ k_1,k_2 \in \NN_{\ge 0} $ be the
  multiplicities of the rays $ (1,1,0,0) $ and $ (0,0,1,1) $ in $C$,
  respectively, and set $ k = \min (k_1,k_2) $.

  If $ C \cdot D < -k $, then the tropical curve $C$ is not realizable in $L$.
\end{Proposition}

\begin{proof}
  Assume without loss of generality that $ k = k_2 $. Note that $ k_2 $ is
  just the length $ m_0-n_0 $ of the right vertical edge of the Newton polytope
  $ \Newt(p^*_0 C) $ as in Lemma \ref{lemma:intprod}. Hence,
    \[ n_0 = m_0 - k_2 = m_0 - k > m_0 + C \cdot D
      = d - m_3, \]
  with the last equation following from Lemma \ref{lemma:interpretation}. So
  $C$ is not realizable in $L$ by Lemma \ref{lemma:intprod}.
\end{proof}

\begin{Example}
  Let $C$ in $L_2^3$ be the tropical curve with
    $$ \P(C) = \{(2,2,0,0),(0,0,2,1),(0,0,0,1)\}. $$
  Then we have $C \cdot D_1 = -1$ by Construction \ref{constr-l32}, so applying
  Proposition \ref{prop:intprod} we see that the tropical curve $C$ is not
  realizable in $L$.
\end{Example}

\begin{Example} \label{ex-intprod}
  Consider the tropical curve $C$ in $L_2^3$ with 
    $$ \P(C) = \{(2,1,0,0),(1,1,0,0),(0,1,1,0),(0,0,1,2),(0,0,1,1)\}. $$
  We have $C \cdot D_1 = -1$, but $C$ is relatively realizable by $f =
  -x_1^2x_2-2x_1x_2^2-x_2^3+x_1^2x_3+x_1x_2x_3 = (x_1+x_2) \cdot
  (-x_1x_2-x_2^2+x_1x_3)$. In this example, we see that the inequality in
  Proposition \ref{prop:intprod} is sharp. Furthermore, it is not a coincidence
  that the polynomial $f$ realizing $C$ is reducible. Either by using our
  algorithm or by applying the obstruction by Brugall\'e and Shaw, we see that
  $C$ is not relatively realizable by an irreducible cycle in $Z_1(E)$.
\end{Example}

\begin{Remark}
  We claimed that Proposition \ref{prop:intprod} and the obstruction given by
  Brugall\'e and Shaw are in fact the same. Example \ref{ex-intprod} shows that
  in case of relative realizability by positive cycles, the given inequality in
  Proposition \ref{prop:intprod} is sharp. Hence, no stronger obstruction is
  possible related to intersection products with classical lines. On the other
  hand, the obstruction by Brugall\'e and Shaw implies Proposition
  \ref{prop:intprod}: Given a tropical curve $C$ in $L_2^3$ fulfilling the
  conditions of Proposition \ref{prop:intprod}, we consider any positive
  decomposition $C = \sum_{i=1}^r C_i + m \cdot D_1$, where $m \leq k$ and $C_i
  \neq D_1$ for all $i = 1,\ldots,r$. Since $C \cdot D_1 < -k$ for $k$ as in
  Proposition \ref{prop:intprod}, we know there is $i \in \{1,\ldots,r\}$ with
  $C_i \neq D_1$ and $C_i \cdot D_1 < 0$. Hence, in any decomposition of $C$,
  there is a component which is not realizable by an irreducible cycle and
  thus, the tropical curve $C$ is not realizable by a positive cycle either.
\end{Remark}

In the proof of Lemma \ref{lemma:intprod} and Proposition \ref{prop:intprod},
we have seen that the non-realizability of the tropical curve $C$ follows from
dependencies between the Newton polytopes $\Newt(f_k)$. In the rest of this
paper it will always be the idea to find suitable dependencies between these
Newton polytopes to prove criteria for realizability. However, we first have to
know how to translate conditions on $C$ to the Newton polytopes $\Newt(p_*^k
(C))$. 

\begin{Lemma} \label{interpretationConditions}
  Let $\{i,j,k,l\} = \{0,1,2,3\}$, and let $C$ be a tropical curve of degree
  $d$ in $L_2^3$ which does not intersect the relative interior of $
  \cone([e_i],[e_j]) $. Then
  \begin{enumerate}
  \item \label{interpretationConditions-a}
    the lattice point corresponding to $x_k^d$ is contained in $\Newt(p_*^l
    (C))$, and
  \item \label{interpretationConditions-b}
    the side of $\Newt(p_*^j (C))$ opposite the vertex corresponding to $
    x_i^d$ has lattice length $m_C(e_i)$.
  \end{enumerate}
\end{Lemma}

\begin{proof}
  By applying a coordinate permutation we may assume that $(i,j,k,l) =
  (0,1,2,3)$, \ie $C$ does not intersect the relative interior of
  $\cone([e_0],[e_1])$. Then:
  \begin{enumerate}
  \item The curve $ p_*^3(C) $ does not intersect $\cone([e_0],[e_1]) $ either,
    so with the notation of Lemma \ref{lemma:newton} it follows
    immediately that $ m=0 $ and thus $ Q_0 = (0,0) \in \Newt(p_*^3 (C)) $.
  \item The ray $ \cone([e_0]) $ (if present) is the only one in $C$ projecting
    to $ \cone([e_0]) $ under $ p^1 $, so we have $ m_C(e_0) = m_{p_*^1 (C)}
    (e_0) $. By definition of the Newton polytope, this is just the
    lattice length of the edge of $ \Newt(p_*^1 (C)) $ opposite the vertex
    corresponding to $ x_0^d $.
  \end{enumerate}
\end{proof}

We are now ready to prove the first new obstruction to realizability. This
obstruction considers tropical curves in $L_2^3$ which are completely contained
in three 2-dimensional cones intersecting in a common face. In the following
picture, they are the shaded cones.

\begin{sidepic}{prop-1} \begin{Proposition} \label{prop:commonray}
  Let $i,j,k,l \in \mathbb N$ with $\{i,j,k,l\} = \{0,1,2,3\}$, and let
  $C\subset L_2^3$ be a tropical curve such that
    \[ C \subset \cone ([e_i],[e_j]) \cup \cone([e_i],[e_k]) \cup
       \cone([e_i],[e_l]) \]
  and $ m_C(e_j) = m_C(e_k) = 0 $. Then $C$ can only be realizable in $L$
  if $m_C(e_l) \neq 1$. 
\end{Proposition} \end{sidepic}

\vspace{-2ex}

\begin{proof}
  By applying a suitable coordinate permutation, we may assume that $i=1$ and
  $l=3$. Assume that $C$ is realizable in $L$ with $m_C(e_3) = 1$. Let $f \in
  K[x_0,x_1,x_2,x_3]$ be homogeneous of degree $d = \deg(C)$ such that $C =
  \Trop(L+(f))$.
  
  First we want to have a closer look at the Newton polytope $\Newt(f_3) =
  \Newt(p_*^3(C))$. Since $C$ does not intersect $\cone([e_0],[e_2])$, we get
  by Lemma \ref{interpretationConditions} \ref{interpretationConditions-a} that
  the lattice point $(0,d)$ corresponding to $x_1^d$ is contained in $
  \Newt(p_*^3 (C))$. Moreover, $C$ does not intersect the relative interiors of
  $\cone([e_0],[e_3])$ and $\cone([e_2],[e_3])$. So $m_C(e_0) = m_C(e_2) = 0$
  implies by Lemma \ref{interpretationConditions}
  \ref{interpretationConditions-b} that $ \Newt(p_*^3(C)) $ meets both $
  \conv((0,0),(0,d))$ and $\conv((0,d),(d,0)) $ in a single point, which must
  therefore be the above point $(0,d)$. Hence, the other points on these two
  lines are not contained in this Newton polytope, \ie using Notation
  \ref{notations:newton} \ref{notations:newton-c} we have
    \[ b_{0,d-r} = b_{r,d-r} = 0 \quad \text{for all $ r=1,\ldots,d $}. \]
  Now consider $\Newt(f_0) = \Newt(p_*^0 (C))$: since $C$ does not intersect
  the relative interior of $\cone([e_2],[e_3])$, we know by Lemma
  \ref{interpretationConditions} \ref{interpretationConditions-a} that the
  lattice point $(d,0)$ corresponding to $x_1^d$ is contained in $\Newt(p_*^0
  (C))$. Moreover, this Newton polytope intersects $\conv((d,0),(0,d))$ with
  lattice length $m_C(e_3) = 1$ by Lemma \ref{interpretationConditions}
  \ref{interpretationConditions-b}. Since $\Newt(p_*^0 (C))$ is convex, this
  means that $(d-r,r) \in \Newt(p_*^0 (C)) \cap \mathbb N^2$ if and only if $r
  \in \{0,1\}$. Hence, we have
  \begin{align*} 
    a_{d-r,r} = 0 &\quad \text{ for all } r=2,\ldots,d, \\
    a_{d-1,1} \neq 0 &.
  \end{align*}
  But we know that $f_0(x_1,x_2,x_3) = f_3(-x_1-x_2-x_3,x_1,x_2)$, so in
  particular it holds that
    $$ f_0(1,-1,0) - f_0(1,0,0) = f_3(0,1,-1) - f_3(-1,1,0). $$
  Considering the coefficients of $f_0$ and $f_3$, this equality results in 
    $$ \sum_{r=1}^d (-1)^r a_{d-r,r}
       = \sum_{r=1}^d (-1)^r b_{0,d-r} - \sum_{r=1}^d (-1)^r b_{r,d-r}. $$
  This is a contradiction, since the above results show that of the terms in
  this equation exactly $ a_{d-1,1} $ is non-zero. Hence, $C$ can only be
  realizable in $L$ if $m_C(e_3) \neq 1$. 
\end{proof}

\begin{Example}
  Let $C$ in $L_2^3$ be the tropical curve with
    $$ \P(C) = \{(4,1,0,0),(0,1,4,0),(0,2,0,3),(0,0,0,1)\}. $$
  Then $C \subset \cone([e_0],[e_1]) \cup \cone([e_1],[e_2]) \cup
  \cone([e_1],[e_3])$ with $ m_C(e_0) =  m_C(e_2) = 0$ and $m_C(e_3) = 1$.
  Applying Proposition \ref{prop:commonray}, we see that $C$ is not realizable
  in $L$.
\end{Example}

A similar obstruction can be proved if the tropical curve is completely
contained in two opposite faces of $L_2^3$.

\vspace{-3ex}

\begin{sidepic}{prop-2} \begin{Proposition} \label{prop:oppositefaces}
  Let $i,j,k,l \in \mathbb N$ such that $\{i,j,k,l\} = \{0,1,2,3\}$, and let
  $C\subset L_2^3$ be a tropical curve such that
    $$ C \subset \cone( [e_i],[e_j] ) \cup \cone([e_k],[e_l] ) $$
  and $m_C(e_i) = m_C(e_j) = m_C(e_k) = 0 $. Then $C$ can only be realizable in
  $L$ if $m_C(e_l) \neq 1$.
\end{Proposition} \end{sidepic}

\vspace{-2ex}

\begin{proof}
  We may assume that $(i,j,k,l) = (0,1,2,3)$. Assume moreover that $C$ is
  realizable in $L$ with $m_C(e_3) = 1$. Let $f \in K[x_0,x_1,x_2,x_3] $ be
  homogeneous of degree $ d=\deg(C) $ such that $C = \Trop(L+(f))$.

  We first consider $\Newt(f_3) = \Newt(p_*^3 (C))$. As $C$ intersects neither
  $\cone([e_0],[e_2])$ nor $ \cone([e_1],[e_2])$, we see by Lemma
  \ref{interpretationConditions} \ref{interpretationConditions-a} that the
  lattice points $(d,0)$ and $(0,d)$ corresponding to $ x_0^d $ and $ x_1^d $,
  respectively, are contained in this Newton polytope. Moreover, part
  \ref{interpretationConditions-b} of this lemma shows that $ \Newt(p_*^3 (C))
  $ intersects both $\conv((0,0),(d,0))$ and $\conv((0,0),(0,d))$ with lattice
  length $0$, hence, in the points $(d,0)$ and $(0,d)$, respectively. So with
  Notation \ref{notations:newton} \ref{notations:newton-c} we have
    $$ b_{r,0} = b_{0,r} = 0 \quad \text{for all $ r = 0,\ldots, d-1 $}. $$
  We now have a closer look at $\Newt(f_0) = \Newt(p_*^0 (C))$. As $C$
  does not intersect the relative interiors of $\cone([e_1],[e_2])$ and
  $ \cone([e_1],[e_3]) $, we have $ (0,0),(0,d) \in \Newt(p_*^0 (C)) $. Moreover,
  we see that this Newton polytope meets $\conv((0,0),(d,0))$ in a single
  point and $\conv((d,0),(0,d))$ with lattice length $1$. Hence,
  \begin{align*} 
    a_{r,0} = 0 &\quad \text{for all $ r=1,\ldots,d $}, \\
    a_{d-r,r} = 0 &\quad \text{for all $ r=0,\ldots,d-2 $}, \\
    a_{1,d-1} \neq 0 &.
  \end{align*}
  Since $f_0(x_1,x_2,x_3) = f_3(-x_1-x_2-x_3,x_1,x_2)$, we have
  \begin{align*}
     & f_0(1,-1,0) - f_0(0,-1,0) - f_0(1,0,-1) + f_0(0,0,-1) \\
    =\;& f_3(0,1,-1) - f_3(1,0,-1) - f_3(0,1,0) + f_3(1,0,0),
  \end{align*}
  so considering the coefficients of $f_0$ and $f_3$ we obtain the
  contradiction
    $$ \sum_{r=0}^{d-1} (-1)^r a_{d-r,r} - \sum_{r=1}^{d} (-1)^{d-r} a_{r,0}
       = \sum_{r=0}^{d-1} (-1)^{d-r} b_{0,r} - \sum_{r=0}^{d-1} (-1)^{d-r}
         b_{r,0} $$
  as all terms except $ a_{1,d-1} $ are zero in this equation. Hence, $C$ can
  only be realizable in $L$ if $m_C(e_3) \neq 1$.
\end{proof}

\begin{Example}
  Let $C$ in $L_2^3$ be the tropical curve with 
    $$ \P(C) = \{(3,1,0,0),(1,3,0,0),(0,0,3,1),(0,0,1,2),(0,0,0,1)\}. $$
  Then $C \subset \cone([e_0],[e_1]) \cup \cone([e_2],[e_3])$ with $m_C(e_i) =
  0$ for $i=0,1,2$ and $m_C(e_3) = 1$, so $C$ is not realizable in $L$ by
  Proposition \ref{prop:oppositefaces}. 
\end{Example}

Our next obstruction to realizability, Proposition \ref{prop:oneside} below,
depends on the characteristic of $K$, since the following preparatory lemma
does so.

\begin{sidepic}{lem-bk} \begin{Lemma} \label{lemma:bogartkatz}
  Let $C \subset L_2^3$ be a tropical curve of degree $d$ in $L_2^3 $, and
  assume that $ \ch(K) = 0$ or $ \ch(K) \ge d $. Moreover, let $ c \in
  \{1,\dots,d-1\} $ and set
  \begin{align*}
    A &= \{(d-k,k) : 0\leq k \leq d, k \neq c \} \\
      &\qquad \cup \; \{ (d-1-k,k) : 0 \leq k \leq d-1 \} \quad \text{and} \\
    B &= \{(0,d-k): 0 \leq k \leq d, k \neq c \},
  \end{align*}
\end{Lemma} \end{sidepic}

  \textit {as indicated with the black dots in the picture on the right.}

  \textit {If $A \cap \Newt(p_*^0(C)) = \emptyset$ and $ |B \cap \Newt(p_*^3
  (C))| = 1$, then $C$ is not realizable in $L$.}

\begin{proof}
  Assume that $C$ is realizable in $L$, and let $ f \in K[x_0,x_1,x_2,x_3] $ be
  homogeneous of degree $d$ with $C = \Trop(L+(f))$. As $ f_0(x_1,x_2,x_3) =
  f_3(-x_1-x_2-x_3,x_1,x_2) $, we have 
  \begin{align*}
    & \frac{\partial f_0}{\partial x_1}(x_1,x_2,x_3) =
      \left( \frac{\partial f_3}{\partial x_1} -
      \frac{\partial f_3}{\partial x_0} \right) (-x_1-x_2-x_3,x_1,x_2) \\
    \text{and} \quad
    & \frac{\partial f_0}{\partial x_3}(x_1,x_2,x_3) =
      - \frac{\partial f_3}{\partial x_0} (-x_1-x_2-x_3,x_1,x_2).
  \end{align*}
  Thus, we get the equality
  \begin{align*}
    & \frac{\partial f_0}{\partial x_1}(1,-1,0)
      - \frac{\partial f_0}{\partial x_3}(1,-1,0)
      - (d-c)\cdot f_0(1,-1,0) \\
    =\;& \frac{\partial f_3}{\partial x_1}(0,1,-1) - (d-c)\cdot f_3(0,1,-1).
  \end{align*}
  Considering the coefficients of $f_0$ and $f_3$ as in Notation
  \ref{notations:newton} \ref{notations:newton-c}, the above equality reads
    \[ \sum_{k=0}^d (-1)^k (c-k) \, a_{d-k,k}
       - \sum_{k=0}^{d-1} (-1)^k a_{d-1-k,k}
       = \sum_{k=0}^d (-1)^k (c-k) \, b_{0,d-k}. \]
  By assumption this simplifies to
    \[ (c-t) \, b_{0,d-t} = 0, \]
  where $ t \in \{0,\dots,d\} $ with $ t \neq c $ is such that $ B \cap
  \Newt(p_*^3(C)) = \{ (0,d-t) \} $. But $ |c-t| < d $, and hence, $ c-t \neq 0
  \in K $ by our assumption on $ \ch (K) $. This means that $ b_{0,d-t} = 0 $.
  But due to the form of $B$ and the fact that the Newton polytope $\Newt(p_*^3
  (C))$ is convex this implies that $ (0,d-t) \notin \Newt(p_*^3 (C)) $, in
  contradiction to our assumption.
\end{proof}

\vspace{-2ex}

\begin{sidepic}{prop-3} \begin{Proposition} \label{prop:oneside}
  Let $ C \subset L_2^3$ be a tropical curve of degree $d$, and assume that
  $\ch(K) = 0$ or $ \ch(K) \ge d $. Moreover, let $ \{i,j,k,l\} = \{0,1,2,3\} $
  such that $C$ intersects neither $\cone([e_i],[e_j])$, nor $ \cone([e_i],
  [e_i+e_k])$, nor $ \cone([e_i],[e_i+e_l]) $. If
    $$ c_1 := \sum_{a e_j + be_k \in \P(C)} a \neq
       \sum_{a e_i + b e_k \in \P(C)} a =: c_2, $$
  and $ 0 < c_2 < d $, then $C$ is not realizable in $L$.
\end{Proposition} \end{sidepic}

\vspace{-1ex}

\begin{proof}
  We may assume that $ (i,j,k,l) = (3,0,2,1) $. The idea of the following proof
  is to show that Lemma \ref{lemma:bogartkatz} can be applied with $ c := c_2
  $. So let us determine the intersection of the Newton polytopes $
  \Newt(p_*^0(C)) $ and $ \Newt(p_*^3 (C)) $ with the sets $A$ and $B$ of this
  lemma, respectively.
  
  As $C$ does not intersect $ \cone([e_0],[e_3])$, we know by Lemma
  \ref{interpretationConditions} \ref{interpretationConditions-b} that
  $\Newt(p_*^3 (C))$ meets $\conv((0,0),(0,d))$ in a single point. Lemma
  \ref{lemma:newton} tells us that this point is $ Q_0 = (0,m) $ with
    \[ m = \sum_{a e_0 + b e_1 \in \P(C)} a
         = d - \sum_{a e_0 + b e_2 \in \P(C)} a = d - c_1, \]
  where the second equality follows from Lemma \ref{lem-degree} together with
  the fact that $C$ does not intersect $ \cone([e_0],[e_3])$. Since $ c_1 \neq
  c_2 = c $ by assumption, this implies that $ B \cap \Newt(p_*^3 (C)) =
  \{(0,d-c_1)\} $, \ie $ |B \cap \Newt(p_*^3 (C))|=1 $.

  In the same way, we see that $ \Newt(p_*^0(C)) $ intersects $
  \conv((d,0),(0,d)) $ in the single point $ (d-c_2,c_2) = (d-c,c) $. We will
  now prove that $ (d-1-k,k) \notin \Newt(p_*^0(C)) $ for all $k=0,\ldots,d-1$.
  As $C$ is of degree $d$ and the Newton polytope is convex, it suffices
  to show that $(d-c-1,c) \notin \Newt(p_*^0 (C))$ and $(d-c,c-1) \notin
  \Newt(p_*^0 (C))$. Let us assume first that $(d-c-1,c) \in \Newt(p_*^0 (C))$.
  Since $(d-(c-1),c-1) \notin \Newt(p_*^0 (C))$ and $\Newt(p_*^0 (C))$ touches 
  $\conv((0,0),(0,d))$, there is an edge of $ \Newt(p_*^0(C)) $ from $(d-c,c)$
  to a point $(a,b)$ with $b \geq c$ and $ a + b < d$. So the Newton polytope
  $\Newt(p_*^0 (C))$ contains an edge with directional vector $(a+c-d,b-c)$.
  This means that 
    $$ \cone([c-b,a+c-d,0]) = \cone([d-a-b,0,d-a-c]) $$ 
  is contained in $ p_*^0(C) $, and thus
    $$ \sigma = \cone([0,d-a-b,0,d-a-c]) $$
  is contained in $C$. But since $ b \geq c $ we have $\sigma \subset
  \cone([e_3],[e_3+e_1])$, in contradiction to the assumption. Similarly, one
  shows that $(d-c,c-1) \notin \Newt(p_*^0(C))$, because otherwise $C$ would
  contain a $1$-dimensional cone in $ \cone([e_3],[e_3+e_2]) $. Altogether this
  means that $ A \cap \Newt(p_*^0(C)) = \emptyset $.

  The statement of the proposition now follows from Lemma
  \ref{lemma:bogartkatz}.
\end{proof}

In \cite{bogartkatz}, Bogart and Katz gave another obstruction to realizability
in the case $ \ch(K)=0 $: If a tropical curve $C$ in $ L^3_2 $ is also
contained in a classical plane $H$ (\ie in a tropical surface which is at the
same time a vector subspace of $ \RR^4/\gen\one $), they proved that $C$ can
only be realizable if it contains a classical line or is a multiple of the
tropical intersection $L_2^3 \cdot H$ (see Lemma \ref{lem-plane}). We will now
reprove this obstruction by applying Proposition \ref{prop:oneside}, thereby
showing that the statement is also true if $ \ch(K) \ge \deg(C) $.

\begin{Proposition} \label{prop:bogartkatz} \cite[Proposition 1.3]{bogartkatz}
  Let $C \subset L_2^3$ be a tropical curve of degree $d$ contained in a
  classical plane $H$, and assume that $ \ch(K)=0 $ or $ \ch(K) \ge d $. If $C$
  is realizable in $L$, then one of the following must hold:
  \begin{enumerate}
  \item \label{prop:bogartkatz-a}
    There is $m \in \mathbb Q_{>0}$ such that $\P(C) = \{ m \cdot v : v \in
    \P(L_2^3 \cdot H)\}$, in which case we call $C$ a multiple of the tropical
    intersection $L_2^3 \cdot H$.
  \item \label{prop:bogartkatz-b}
    The tropical intersection $L_2^3 \cdot H$ contains a classical line.
  \end{enumerate}
\end{Proposition}

\begin{proof}
  As $H$ is a tropical variety, there is $a \in \mathbb Z^4 \setminus \{0\}$
  with $ a_0+a_1+a_2+a_3 = 0$ such that $H = \{ [x] \in \mathbb R^4 / \gen\one:
  a \cdot x = 0 \}$. We assume that neither the tropical intersection $ L_2^3
  \cdot H$ contains a classical line nor $C$ is a multiple of $ L_2^3 \cdot H$,
  and show that in this case $C$ is not realizable in $L$. The cases of
  transversal and non-transversal intersection of $L^3_2$ with $H$ will in the
  following be handled separately.

  If $L_2^3$ and $H$ do not intersect transversally, then $H$ contains a
  maximal cone of $L_2^3$. Without loss of generality we may assume $ H = \{
  [x] \in \mathbb R^4/\gen\one : x_2 = x_3 \}$. Then the geometric intersection
  is given by
    $$ L_2^3 \cap H = \cone([e_0],[e_1]) \cup \cone([e_2+e_3]), $$
  so $\P(C)$ has the form $\P(C) = \{(b_1,c_1,0,0),\ldots,(b_r,c_r,0,0),
  (0,0,d,d)\}$, where $d = \deg(C)$ and $\sum_k b_k = \sum_k c_k = d $. Since
  $C$ was assumed not to be a multiple of the tropical intersection
    $$ L_2^3 \cdot H = \cone([e_0]) \cup \cone([e_1]) \cup \cone([e_2+e_3])
    $$
  (see Lemma \ref{lem-plane}), there is $k \in \{1,\ldots,r\}$ with $a_k > 0$
  and $b_k > 0$. Let $D_1 \subset L_2^3$ be the classical line with $\P(D_1) =
  \{(1,1,0,0),(0,0,1,1)\}$. Then by Construction \ref{constr-l32} we have $C
  \cdot D_1 = d - \sum_{k=1}^r \min\{a_k,b_k\} - d < 0$, and since $C$ does not
  contain $D_1$ by assumption, the tropical curve $C$ is not realizable in $L$
  by Proposition \ref{prop:intprod}.

  We now assume that $L_2^3$ and $H$ intersect transversally, so that $C$ is
  contained in $ D := L_2^3 \cdot H $ as a set. By possibly replacing $a$ by
  $-a$ we may assume that $|\{i: a_i > 0 \}| \leq |\{ i : a_i < 0 \}|$. In
  particular, we have $|\{ i : a_i > 0 \}| \in \{1,2\}$. We consider the
  following two cases:

  \textbf{Case 1:} $|\{i: a_i > 0\}| = 1$, without loss of generality $ a_0 > 0
  $. By Lemma \ref{lem-plane} we know that 
    $$ \P(D) = \{ (-a_1,a_0,0,0),(-a_2,0,a_0,0),(-a_3,0,0,a_0)\}. $$
  But the only balanced curves supported on these three rays are multiples of
  $D$. So we arrive at a contradiction to our assumption that $C$ is not such a
  multiple.

  \textbf{Case 2:} $|\{i: a_i > 0 \}| = 2$, and hence, also $|\{i: a_i < 0 \}|
  = 2$. By a coordinate permutation and possibly replacing $a$ by $-a$ we can
  assume that $ |a_0| = \max \{ |a_i|: i=0,\dots,3\} $ as well as $ a_0,a_1>0 $
  and $ a_2,a_3<0 $. Then Lemma \ref{lem-plane} tells us that
    \[ P(D) =
        \{ (-a_2,0,a_0,0),(-a_3,0,0,a_0),(0,-a_2,a_1,0),(0,-a_3,0,a_1)\}, \]
  and thus $\P(C)$ is given by
    \[ P(C) = \left\{\lambda_{02} (-a_2,0,a_0,0), \lambda_{03} (-a_3,0,0,a_0),
         \lambda_{12} (0,-a_2,a_1,0), \lambda_{13} (0,-a_3,0,a_1) \right\} \]
  for some $\lambda_{ij} \in \mathbb Q_{\ge 0} $, where by abuse of notation we
  allow some of these coefficients to be zero. Note however that $ \lambda_{02}
  $ cannot be zero: otherwise Lemma \ref{lem-degree} implies
    \[ - \lambda_{03} a_3 = d = \lambda_{03} a_0 + \lambda_{13} a_1 \]
  by adding the $0$-th and last coordinates of these vectors, respectively.
  This means that $ \lambda_{03} (a_0+a_3) + \lambda_{13} a_1 = 0 $, which is
  only possible if $ \lambda_{13}=0 $ as well because $ a_0+a_3 \ge 0 $ and
  $ a_1 > 0 $. But then $C$ consists of only two rays and can thus only be
  balanced if it is a multiple of a classical line, which we excluded. This
  contradiction shows that $ \lambda_{02} > 0 $. Of course, by symmetry we
  then get $ \lambda_{03} > 0 $ as well.

  We want to use Proposition \ref{prop:oneside} with $ (i,j,k,l)=(0,1,2,3) $ to
  show that $C$ is not realizable in $L$. So let us check that the assumptions
  of this proposition are met. It is obvious that $C$ does not intersect $
  \cone([e_0],[e_1]) $. We now claim that $C$ meets neither $
  \cone([e_0],[e_0+e_2])$ nor $\cone([e_0],[e_0+e_3])$. This is equivalent to
  claiming that $ a_0+a_2>0 $ and $ a_0+a_3>0 $. Assume this is not the case,
  so without loss of generality we assume $ a_0+a_2=0 $. But then $D$ contains
  the classical line $ \Span([1,0,1,0]) $, which is a contradiction to our
  assumption.

  So, to apply Proposition \ref{prop:oneside}, the only thing left to show is
  that with
    \[ c_1 := \sum_{ae_1+be_2 \in \P(C)} a = -\lambda_{12} a_2
       \quad\text{and}\quad
       c_2 := \sum_{ae_0+be_2 \in \P(C)} a = -\lambda_{02} a_2 \]
  we have $ c_1 \neq c_2 $ and $ c_2 \notin \{0,d\} $. If $c_1 = c_2$ we get
  $ \lambda_{02} = \lambda_{12} $, and hence, by balancing also $ \lambda_{02}
  = \lambda_{03} = \lambda_{12} = \lambda_{13} $. In this case $C$ would be a
  multiple of $D$, which we excluded. As $ \lambda_{02} > 0 $ we have $ c_2
  \neq 0 $, and since $ d = c_2 - \lambda_{03} a_3 $ by Lemma
  \ref{lem-degree} and $ \lambda_{03} > 0 $, we also have $ c_2 \neq d $.
  So $C$ fulfills all the conditions of Proposition \ref{prop:oneside} and thus
  is not realizable in $L$.
\end{proof}

\begin{Example}
  Proposition \ref{prop:oneside} is more general than the obstruction by
  Bogart and Katz in \cite{bogartkatz}, even in characteristic $0$. Consider
  for instance the tropical curve $C$ of degree $5$ in $L_2^3$ with 
    $$ \P(C) = \{(1,0,2,0),(2,0,3,0),(0,1,0,2),(0,1,0,3),(2,3,0,0)\}. $$
  It does not lie in a classical plane, so the obstruction of Bogart and Katz
  cannot be applied. But $C$ is not realizable in $L$ by Proposition
  \ref{prop:oneside} with $ (i,j,k,l)=(0,3,1,2) $ if $ \ch(K)=0 $ or $ \ch(K)
  \ge 5 $: we have $ c_1 = 5 $, $ c_2 = 2 $, and
    \[ C \cap \cone([e_0],[e_3]) = C \cap \cone([e_0],[e_0+e_1])
       = C \cap \cone([e_0],[e_0+e_2]) = \emptyset. \]
\end{Example}

\begin{example}[Non-realizable curves of small degree] \label{ex-table}
  For characteristic $0$, the following table contains a complete list, as
  obtained by Algorithm \ref{algorithm}, of all tropical curves in $L_2^3$ (up
  to symmetry by coordinate permutations) which have degree at most $3$ and are
  not realizable in $L$. To get a feeling of the power of each obstruction
  presented, we always indicate by which obstruction the non-realizability may
  be proved.

  \pagebreak[3]
  \begin{center} \begin{footnotesize} \begin{longtable}{||c||c|c|c|c|c||}
    $\P(C)$ & Brugall\'e-Shaw & \ref{prop:intprod} &
              \ref{prop:commonray} & \ref{prop:oppositefaces} &
              \ref{prop:oneside} \\ \hline
    $ \{(2,2,0,0),(0,0,2,1),(0,0,0,1)\} $ &
      \checkmark &\checkmark&&\checkmark& \\
    $ \{(3,3,0,0),(0,0,3,2),(0,0,0,1)\} $ &
      \checkmark &\checkmark&&\checkmark& \\
    $ \{(3,3,0,0),(0,0,3,1),(0,0,0,2)\} $ &
      \checkmark&\checkmark&&& \\
    $ \{(3,3,0,0),(0,0,2,1),(0,0,1,2)\} $ &
      \checkmark&\checkmark&&& \\
    $ \{(3,3,0,0),(0,0,2,1),(0,0,1,1),(0,0,0,1)\} $ &
      \checkmark&\checkmark&&\checkmark& \\
    $ \{(3,3,0,0),(0,0,2,1),(0,0,1,0),(0,0,0,2)\} $ &
      \checkmark&\checkmark&&& \\
    $ \{(3,2,0,0),(0,1,1,0),(0,0,2,2),(0,0,0,1)\} $ &
      \checkmark&\checkmark&&& \\
    $ \{(3,2,0,0),(0,1,0,2),(0,0,3,1) \} $ &
      \checkmark&&&&\checkmark \\
    $ \{(3,2,0,0),(0,1,0,1),(0,0,3,2) \} $ &
      \checkmark&\checkmark&&& \\
    $ \{(3,2,0,0),(0,1,0,1),(0,0,2,1),(0,0,1,1) \} $ &
      \checkmark&\checkmark&&& \\
    $ \{(3,2,0,0),(0,1,0,0),(0,0,3,2),(0,0,0,1) \} $ &
      \checkmark&\checkmark&&& \\
    $ \{(3,2,0,0),(0,1,0,0),(0,0,2,2),(0,0,1,0),(0,0,0,1) \} $ &
      \checkmark&\checkmark&&& \\
    $ \{(3,2,0,0),(0,1,0,0),(0,0,2,1),(0,0,1,2) \} $ &
      \checkmark&\checkmark&&\checkmark& \\
    $ \{(3,2,0,0),(0,1,0,0),(0,0,2,1),(0,0,1,1),(0,0,0,1)\} $ &
      \checkmark&\checkmark&&& \\
    $ \{(3,1,0,0),(0,1,3,0),(0,1,0,2),(0,0,0,1)\} $ &
      &&\checkmark&& \\
    $ \{(2,2,0,0),(1,0,0,0),(0,1,0,0),(0,0,2,1),(0,0,1,2) \} $ &
      \checkmark&\checkmark&&& \\
    $ \{(2,1,0,0),(1,2,0,0),(0,0,2,1),(0,0,1,2) \} $ &
      \checkmark&\checkmark&&& \\
    $ \{(2,1,0,0),(1,2,0,0),(0,0,2,1),(0,0,1,1),(0,0,0,1) \} $ &
      \checkmark&\checkmark&&\checkmark&
  \end{longtable} \end{footnotesize} \end{center}

  In the list above, we indicated that the non-realizability of the tropical
  curve $C$ with $\P(C) = \{(3,2,0,0),(0,1,0,2),(0,0,3,1) \}$ can be proved
  using the obstruction by Brugall\'e-Shaw. This results from the fact that $C
  \cdot D < 0$, where $D$ is the tropical curve in $L_2^3$ with $\P(D) =
  \{(6,4,0,0),(0,2,5,0),(0,0,1,3),(0,0,0,3)\}$. Using for instance Algorithm
  \ref{algorithm}, one sees that $D$ is realizable in $L$.

  Contrary to this observation, one can show that if $C$ is the tropical curve
  in $L_2^3$ with $\P(C) = \{(3,1,0,0),(0,1,3,0),(0,1,0,2),(0,0,0,1)\}$, then
  $C \cdot D \geq 0$ for all tropical curves $D$ in $L_2^3$. Hence, the
  obstruction by Brugall\'e-Shaw cannot be used to prove the non-realizability
  of $C$. 

  In degree $4$ there are 138 tropical curves (up to coordinate permutations)
  in $L_2^3$ which are not realizable in $L$ if $\ch(K) = 0$. We see in the
  list above that Proposition \ref{prop:intprod} is a strong obstruction. In
  fact, among these $138$ non-realizable degree-$4$ curves in $L_2^3$ there are
  only $22$ curves whose non-realizability cannot be proved using Proposition
  \ref{prop:intprod}. We are now listing these $22$ curves with indications
  which obstruction can be used to prove the non-realizability.

  \pagebreak[3]
  \begin{center} \begin{footnotesize} \begin{longtable}{||c||c|c|c||}
    $\P(C)$ & \ref{prop:commonray} &
              \ref{prop:oppositefaces} &
              \ref{prop:oneside} \\ \hline
    $ \{ (4,3,0,0),(0,1,0,3),(0,0,4,1) \} $ &
      && \checkmark \\
    $ \{ (4,3,0,0),(0,1,0,3),(0,0,3,1),(0,0,1,0)\} $ &
      && \checkmark \\
    $ \{ (4,3,0,0),(0,1,0,2),(0,0,4,1),(0,0,0,1)\} $ &
      &&\checkmark \\
    $ \{ (4,2,0,0),(0,2,0,3),(0,0,4,1) \} $ &
      &&\checkmark \\
    $ \{ (4,2,0,0),(0,2,0,2),(0,0,3,2),(0,0,1,0)\} $ &
      && \\
    $ \{ (4,2,0,0),(0,2,0,2),(0,0,3,1),(0,0,1,1)\} $ &
      && \\
    $ \{ (4,2,0,0),(0,2,0,1),(0,0,4,2),(0,0,0,1)\} $ &
      && \\
    $ \{ (4,2,0,0),(0,1,4,0),(0,1,0,3),(0,0,0,1)\} $ &
      \checkmark&& \\
    $ \{ (4,1,0,0),(0,3,0,2),(0,0,3,2),(0,0,1,0)\} $ &
      && \\
    $ \{ (4,1,0,0),(0,3,0,1),(0,0,3,2),(0,0,1,1)\} $ &
      && \\
    $ \{ (4,1,0,0),(0,2,0,3),(0,1,4,0),(0,0,0,1)\} $ &
      \checkmark&& \\
    $ \{ (4,1,0,0),(0,1,4,0),(0,1,0,3),(0,1,0,0),(0,0,0,1)\} $ &
      \checkmark&& \\
    $ \{ (4,1,0,0),(0,1,4,0),(0,1,0,2),(0,1,0,1),(0,0,0,1)\} $ &
      \checkmark&& \\
    $ \{ (4,1,0,0),(0,1,3,0),(0,1,1,0),(0,1,0,3),(0,0,0,1)\} $ &
      \checkmark&& \\
    $ \{ (3,2,0,0),(1,1,0,0),(0,1,0,2),(0,0,3,1),(0,0,1,1) \} $ &
      && \\
    $ \{ (3,2,0,0),(1,0,3,0),(0,2,0,3),(0,0,1,1)\} $ &
      && \checkmark \\
    $ \{ (3,2,0,0),(1,0,2,0),(0,2,0,3),(0,0,1,1),(0,0,1,0)\} $ &
      && \checkmark\\
    $ \{ (3,2,0,0),(1,0,2,0),(0,1,0,2),(0,1,0,1),(0,0,2,1)\} $ &
      && \\
    $ \{ (3,2,0,0),(1,0,0,3),(0,2,3,0),(0,0,1,1)\} $ &
      && \checkmark\\
    $ \{ (3,2,0,0),(1,0,0,2),(0,2,1,0),(0,0,2,1),(0,0,1,1)\} $ &
      && \\
    $ \{ (3,2,0,0),(1,0,0,1),(0,2,1,0),(0,0,3,2),(0,0,0,1)\} $ &
      && \\
    $ \{ (3,1,0,0),(1,3,0,0),(0,0,3,1),(0,0,1,2),(0,0,0,1)\} $ &
      &\checkmark& \\
  \end{longtable} \end{footnotesize} \end{center}
\end{example}

\begin{example}[Realizability depends on $ \ch(K) $] \label{ex-char}
  Some of our criteria for realizability in this section had a dependence on
  the characteristic of $K$. The following example shows that the relative
  realizability of a tropical curve in $ L^3_2 $ may in fact depend on $ \ch(K)
  $: consider the curves $C,D$ in $L_2^3$ with
  \begin{align*}
    \P(C) &= \{(0,0,3,1),(0,1,0,2),(3,2,0,0)\}, \\
    \P(D) &= \{(0,0,1,0),(0,0,2,1),(0,1,0,2),(3,2,0,0)\}.
  \end{align*}
  For $ \ch(K)=0 $ we see in the list above that $C$ is not realizable in $L$,
  but $D$ is. However, using our algorithm we get for $\ch(K) = 2$ that $C$ is
  realizable in $L$, while $D$ is not. Hence, the realizability in $L$ depends
  on the characteristic of $K$.
\end{example}